\documentclass[11pt,paper=a4,DIV=12,abstract=true]{scrartcl}

\usepackage[T1]{fontenc}
\usepackage[utf8]{inputenc}
\usepackage{lmodern}
\usepackage{microtype}

\usepackage{amsmath,amssymb,amsthm,mathtools}
\usepackage{bm,dsfont,comment}
\usepackage[english]{babel}
\usepackage[shortlabels]{enumitem}
\DeclareUnicodeCharacter{00A0}{ }

\usepackage{graphicx}

\usepackage{aliascnt}
\usepackage[dvipsnames]{xcolor}

\usepackage[round,authoryear]{natbib}
\usepackage[unicode=true]{hyperref}
\usepackage[nameinlink,noabbrev,capitalize]{cleveref}
\setkomafont{disposition}{\normalfont\bfseries}
\addtokomafont{section}{\large}
\addtokomafont{subsection}{\normalsize}
\setkomafont{title}{\normalfont\bfseries}
\newcommand{\manuscripttitle}[1]{\title{{\Large #1}}}
\setkomafont{author}{\normalfont}
\setkomafont{date}{\normalfont}
\setkomafont{caption}{\small}
\setkomafont{captionlabel}{\small\bfseries}

\numberwithin{equation}{section}

\theoremstyle{plain}
\newtheorem{theorem}{Theorem}[section]

\newaliascnt{proposition}{theorem}
\newtheorem{proposition}[proposition]{Proposition}
\aliascntresetthe{proposition}

\newaliascnt{lemma}{theorem}
\newtheorem{lemma}[lemma]{Lemma}
\aliascntresetthe{lemma}

\newaliascnt{corollary}{theorem}

\aliascntresetthe{corollary}

\theoremstyle{definition}

\newaliascnt{assumption}{theorem}
\newtheorem{assumption}[assumption]{Standing Assumption}
\aliascntresetthe{assumption}

\newaliascnt{definition}{theorem}
\newtheorem{definition}[definition]{Definition}
\aliascntresetthe{definition}

\newaliascnt{example}{theorem}
\newtheorem{example}[example]{Example}
\aliascntresetthe{example}

\theoremstyle{remark}

\newaliascnt{remark}{theorem}
\newtheorem{remark}[remark]{Remark}
\aliascntresetthe{remark}

\crefname{theorem}{Theorem}{Theorems}
\Crefname{theorem}{Theorem}{Theorems}
\crefname{proposition}{Proposition}{Propositions}
\Crefname{proposition}{Proposition}{Propositions}
\crefname{lemma}{Lemma}{Lemmas}
\Crefname{lemma}{Lemma}{Lemmas}
\crefname{corollary}{Corollary}{Corollaries}
\Crefname{corollary}{Corollary}{Corollaries}
\crefname{assumption}{Assumption}{Assumptions}
\Crefname{assumption}{Assumption}{Assumptions}
\crefname{definition}{Definition}{Definitions}
\Crefname{definition}{Definition}{Definitions}
\crefname{example}{Example}{Examples}
\Crefname{example}{Example}{Examples}
\crefname{remark}{Remark}{Remarks}
\Crefname{remark}{Remark}{Remarks}

\newcommand{\na}{\mathbb{N}}
\newcommand{\re}{\mathbb{R}}
\newcommand{\replus}{\re_{+}}
\newcommand{\Fcal}{\mathcal{F}}
\newcommand{\FF}{\mathbb{F}}
\newcommand{\PP}{\mathbb{P}}
\newcommand{\QQ}{\mathbb{Q}}
\newcommand{\B}{\mathcal{B}}
\newcommand{\T}{\mathcal{T}}
\newcommand{\from}{{:}\penalty\binoppenalty\mskip\thickmuskip}
\newcommand{\indicatorset}[1]{\mathds{1}_{#1}}

\makeatletter
\let\old@fnsymbol\@fnsymbol
\renewcommand{\@fnsymbol}[1]{%
  \ifnum#1=2
    \TextOrMath{\textasteriskcentered\textasteriskcentered}{**}%
  \else
    \old@fnsymbol{#1}%
  \fi
}
\makeatother

\makeatletter

\newcommand{\makeoverbar}[7]{%
  \setbox0=\hbox{$\m@th#2\mkern#5mu{{}#3{}}\mkern#6mu$}%
  \setbox1=\null \dimen@=#4\fontdimen8#13 \dimen@=3.5\dimen@
  \advance\dimen@ by \ht0 \dimen@=-#7\dimen@ \advance\dimen@ by \wd0
  \ht1=\ht0 \dp1=\dp0 \wd1=\dimen@
  \dimen@=\fontdimen8#13 \fontdimen8#13=#4\fontdimen8#13
  \rlap{\hbox to \wd0{$\m@th\hss#2{\overline{\box1}}\mkern#5mu$}}
  \fontdimen8#13=\dimen@}
\makeatother

\allowdisplaybreaks
\DeclareMathOperator*{\argmin}{arg\,min}

\newcommand{\keywords}[1]{\par\medskip\noindent\textbf{Keywords.} #1}
\newcommand{\msc}[1]{\par\smallskip\noindent\textbf{MSC 2020.} #1}
\newcommand{\jel}[1]{\par\smallskip\noindent\textbf{JEL classification.} #1}

\hypersetup{%
  hidelinks,
  pdftitle={Algorithmic Trading and Stochastic Integration},
  pdfauthor={Aleksandar Arandjelović and Uwe Schmock},
  pdfkeywords={Neural networks; universal approximation theorems; Orlicz spaces; stochastic integration; semimartingales; mean-variance hedging; equivalent martingale measures}
}

\manuscripttitle{Algorithmic Trading and Stochastic Integration}

\author{%
  Aleksandar Arandjelović\thanks{Department of Mathematics, ETH Zurich, Zurich, Switzerland, and Institute for Statistics and Mathematics, WU Vienna University of Economics and Business, Vienna, Austria. Email: \href{mailto:aleksandar.arandjelovic@wu.ac.at}{aleksandar.arandjelovic@wu.ac.at}. Corresponding author.}
  \and
  Uwe Schmock\thanks{Institute of Statistics and Mathematical Methods in Economics, TU Wien, Vienna, Austria. Email: \href{mailto:uwe.schmock@tuwien.ac.at}{uwe.schmock@tuwien.ac.at}.}
}

\date{September 25, 2026}

\begin{document}

\maketitle 

\begin{abstract}
We study simple predictable processes whose coefficients are represented by neural networks.
On finite measure spaces, we establish density results for neural networks in Orlicz spaces.
For filtrations generated by a stochastic process, measurable random variables, including at stopping times, can be approximated by neural networks depending on finitely many observations.
Every stochastic integral with respect to a semimartingale can then be approximated, in the semimartingale topology, by integrals of such simple predictable processes.
We show that restricting trading strategies to this class leaves the minimal mean-variance hedging error under partial information unchanged and obtain a no-free-lunch characterization in terms of equivalent martingale measures.
Finally, the Bichteler--Dellacherie characterization of semimartingales remains valid even upon restricting the predictable integrands to those whose coefficients are represented by neural networks.
\end{abstract}

\keywords{Neural networks; universal approximation theorems; Orlicz spaces; stochastic integration; semimartingales; mean-variance hedging; equivalent martingale measures}
\msc{Primary 60H05; Secondary 41A65, 46E30, 68T07, 91G20}
\jel{C45; G12; G13}

\clearpage

\section{Introduction}\label{sec:introduction}
Algorithmic trading refers to the use of computer algorithms, with limited or no human intervention, to automatically determine whether and when to submit orders for financial instruments, specify their prices and quantities, and manage them after submission \citep{ecb19algorithmic}.
Since the early 2000s, algorithmic trading has expanded steadily; a 2026 supervisory briefing describes professional trading without computer algorithms as exceptional \citep{esma26algorithmic}.
This expansion has been facilitated by greater computing power, lower storage costs, and advances in machine learning.
Automation can also reduce labor and related costs, process large volumes of market data rapidly, and execute trades promptly \citep{ecb19algorithmic}.
These benefits are accompanied by operational risks from failures of algorithms, information systems, or trading processes, and by the risk that interactions between trading algorithms amplify market stress \citep{ecb19algorithmic,cftcsec10market}.
The Flash Crash of May~6, 2010 illustrates the latter risk: a large automated sell program entered an already stressed market and interacted with high-frequency traders as liquidity deteriorated rapidly \citep[pp.~2--3,~6]{cftcsec10market}.

Research on artificial neural networks was recognized by the 2024 Nobel Prize in Physics awarded to John Hopfield and Geoffrey Hinton \citep{nobel24physics}.
The methods used to train neural networks draw on classical mathematics: backpropagation uses the chain rule from differential calculus \citep{schmidhuber15overview}, and \citet{cauchy47methode} proposed an early gradient descent method.
Universal approximation results for neural networks were established beginning in the late 1980s \citep{cybenko89approximation,hornik89approximation,hornik91approximation}.
For example, consider feedforward networks on $\re^d$ with a continuous, bounded, nonconstant activation function, one hidden layer with finitely many units, a linear output, and unrestricted weights and biases.
Their restrictions to any compact $K\subseteq\re^d$ are dense in $C(K)$ under the uniform norm.
They are also dense in $L^p(\mu)$ for every finite Borel measure $\mu$ on $\re^d$ and $1\le p<\infty$ \citep[Theorems~1--2]{hornik91approximation}.
Such density results are called \emph{universal approximation theorems}, and the underlying density property is called the \emph{universal approximation property} \citep{kratsios21approximation}.

In mathematical finance, prices and trading strategies are usually modeled by stochastic processes, together with a filtration describing the flow of information.
Uniform approximation in $C(K)$ then has a natural stochastic analogue in uniform convergence on compact time intervals in probability (ucp).
For real-valued càdlàg or càglàd processes $X^n$ and $X$ on $[0,\infty)$, this means that $\sup_{0\le s\le t}|X_s^n-X_s|\to0$ in probability for every $0<t<\infty$.
A finer topology than the ucp topology for stochastic integral processes is the semimartingale topology of \citet[p.~264]{emery79topologie}, denoted by $\T_{\mathrm{sm}}$.
For stochastic integrands, convergence in $L^p(M)$ is an analogue of classical $L^p(\mu)$-approximation.
For a real-valued càdlàg martingale $M$ with $M_0=0$ and $\sup_{t\ge0}|M_t|\in L^p$, this space consists of predictable processes $V$ with finite seminorm $\|V\|_{L^p(M)}\coloneqq\mathbb{E}[(\int_0^\infty V_s^2\,\mathrm{d}[M]_s)^{p/2}]^{1/p}$.
Here $[M]$ denotes quadratic variation, so the error is measured against $[M]$, rather than uniformly in time.
By the Burkholder--Davis--Gundy inequality, convergence of predictable integrands in this seminorm yields convergence of the corresponding stochastic integrals in the maximal $L^p$ norm \citep[Theorem~11.5.5 and Corollary~12.3.6]{cohen15calculus}.
Thus extending classical approximation theorems to stochastic processes requires more than changing the topology: approximation errors are random, approximating integrands must be predictable, and their errors must be controlled relative to the integrator.

Our starting point is the discrete-time deep hedging framework of \citet{buehler19hedging}, which accommodates transaction costs and liquidity constraints.
Its approximation argument expresses trading positions as measurable functions of market information and then approximates these functions by neural networks.
We study the corresponding approximation problem in frictionless continuous-time markets, using algorithmic strategies built from buy-and-hold positions specified by neural networks.
The driving process $Y$ is an $\re^d$-valued L\'{e}vy process with its completed natural filtration $\FF=(\Fcal_t)_{t\ge0}$.
A buy-and-hold component then has the form
\begin{equation*}
f(\tau_i, Y_{t_1\wedge\tau_i}, \ldots, Y_{t_n\wedge\tau_i})\indicatorset{(\tau_i,\tau_{i+1}]}(t),
\end{equation*}
where $\tau_i\le\tau_{i+1}<\infty$ are $\FF$-stopping times, $0\le t_1<\cdots<t_n<\infty$ are deterministic sampling times, and $f\from\re^{1+dn}\to\re$ is a neural network.
At $\tau_i$, the neural network $f$ uses finitely many observed values of $Y$, $Y_{t_{1}\wedge\tau_i}, \ldots, Y_{t_n\wedge\tau_i}$, and specifies the position $f(\tau_i, Y_{t_{1}\wedge\tau_i}, \ldots, Y_{t_n\wedge\tau_i})$ held until $\tau_{i+1} > \tau_i$. 
Algorithmic strategies are thus simple predictable processes, which are the elementary integrands used in the construction of stochastic integrals.

Just as measure theory begins with simple functions and extends the elementary integral to more general functions by passing to a limit, stochastic integration is built from simple predictable processes and their stochastic integrals.
The following theorem shows that algorithmic strategies can play the role of these elementary integrands (see Section~\ref{sec:stochastic-integration} for a proof):
\begin{theorem}\label{thm:ucp-approximation-intro}
Let $X$ be a c\`{a}dl\`{a}g $\FF$-semimartingale and let $V$ be an $\FF$-predictable $X$-integrable process.
Then there exists a sequence $(V^n)_{n\in\na}$ of algorithmic strategies such that
\begin{equation*}
\int V^n\,\mathrm{d}X \to \int V\,\mathrm{d}X,\quad n\to\infty,
\end{equation*}
where convergence holds in $\T_{\mathrm{sm}}$.
If $V$ is c\`{a}gl\`{a}d, then the sequence $(V^n)_{n\in\na}$ can be chosen so that $V^n \to V$ in the ucp topology.
\end{theorem}
Algorithmic strategies form a dense approximation class not only for trading strategies themselves, but also for the stochastic integrals they generate.
Theorem~\ref{thm:ucp-approximation-intro} thus gives a neural network approximation in the theory of stochastic integration, extends the setting of \citet{buehler19hedging} to continuous time, and yields several consequences for mathematical finance.

Hedging is a standard application of trading algorithms, both for managing option exposures and for offsetting positions across related markets \citep[pp.~33,~40]{sec20algorithmic}.
It is also a central problem of mathematical finance, where mean-variance hedging provides a classical way to address the residual risk of incomplete markets \citep{rheinlaender97projections,schweizer01quadratic}.
For a square-integrable martingale price process satisfying the moment assumption of Theorem~\ref{thm:mean-variance-hedging}, stochastic integrals of algorithmic strategies are dense in the stable subspace it generates, with convergence in $\mathcal{H}^2$.
The expected squared hedging error of algorithmic strategies can therefore be made arbitrarily close to the minimum mean-square hedging error, thereby extending the corresponding result of \citet{buehler19hedging} to continuous time.

The fundamental theorem of asset pricing connects the no-arbitrage principle with martingale methods underlying risk-neutral valuation.
For locally bounded semimartingale price processes, the absence of a free lunch with vanishing risk is equivalent to the existence of an equivalent local martingale measure \citep{delbaen94arbitrage}.
A related classical approach works directly with gains from elementary trading strategies: for continuous price processes, \citet[Theorem~3]{stricker90arbitrage} gives an $L^p$ criterion for the existence of an equivalent martingale measure.
We adapt this criterion to algorithmic strategies by showing that their terminal gains have the same $L^p$-closure as those of elementary strategies, under the moment assumptions of Theorem~\ref{thm:equivalent-martingale-measures}.
This allows the corresponding $L^p$ no-free-lunch condition to be formulated using the neural network gain class, rather than all elementary gains.

The Bichteler--Dellacherie theorem ties the semimartingale property directly to the stability of stochastic integration: an adapted c\`{a}dl\`{a}g process is a semimartingale precisely when it is a good integrator for bounded elementary predictable processes \citep{bichteler79integrators,dellacherie80probabilites,bichteler81integration}.
Outside the class of semimartingales, models can exhibit pathological trading behavior: fractional Brownian motion, for instance, fails to be a semimartingale unless its Hurst parameter equals $1/2$, and the corresponding frictionless price models can admit arbitrage when trading times are allowed to become arbitrarily close \citep{cheridito03fractional}.
The Bichteler--Dellacherie theorem is relevant here because its test class consists of elementary strategies, which include our algorithmic strategies.
We prove that the semimartingale property can be characterized in terms of algorithmic strategies: a c\`{a}dl\`{a}g adapted process is a semimartingale if and only if it is a good integrator for bounded algorithmic strategies, thus allowing one to test the semimartingale property within this class.

Our contributions are as follows.
\begin{enumerate}[label=\Roman*.]
\item We establish a universal approximation theorem for neural networks on measurable spaces, with approximation measured in Orlicz norms (Theorem~\ref{thm:orlicz-approximation}), and examples including neural networks on spaces of c\`{a}dl\`{a}g paths and martingales.
\item We extend this approximation principle to random variables measurable with respect to $\sigma$-algebras generated by arbitrary, possibly uncountable, families of random variables (Theorem~\ref{thm:random-variable-approximation}).
The approximating neural networks depend on only finitely many members of the generating family.
For filtrations generated by a L\'{e}vy process, this extends to random variables measurable with respect to stopping-time $\sigma$-algebras.
When the filtration is generated by either a continuous or a c\`{a}dl\`{a}g process, we also prove the corresponding approximation result for neural networks that depend on the whole path of the process.
\item We extend these results to stochastic integrals, proving that every predictable stochastic integral with respect to a semimartingale can be approximated in $\T_{\mathrm{sm}}$ by integrals of algorithmic strategies. For adapted c\`{a}gl\`{a}d integrands, algorithmic strategies also approximate the integrands in ucp (Theorem~\ref{thm:ucp-approximation}). Under stronger integrability assumptions, we obtain approximation in $L^p(X)$, with convergence of the integrals in $\mathcal H_{\mathrm{sm}}^p$ (Theorem~\ref{thm:lp-integrand-approximation}).
\item We obtain an approximation result for continuous-time mean-variance hedging with algorithmic strategies (Theorem~\ref{thm:mean-variance-hedging}), an $L^p$ criterion for the existence of equivalent martingale measures in terms of algorithmic strategies (Theorem~\ref{thm:equivalent-martingale-measures}), and a characterization of semimartingales in terms of algorithmic strategies (Theorem~\ref{thm:semimartingale-characterization}).
\end{enumerate}

An earlier version of the main results of this paper appeared in Chapter~5 of the first author's doctoral thesis \citep{arandjelovic24theory}.

The remainder of the paper is organized as follows: 
Section~\ref{sec:networks-measurable-spaces} establishes universal approximation in Orlicz spaces for functions on measurable spaces. 
Section~\ref{sec:random-variable-approximation} proves universal approximation for random variables measurable with respect to $\sigma$-algebras generated by arbitrary, possibly uncountable, families.
Section~\ref{sec:stochastic-integration} introduces algorithmic strategies, proves universal approximation for stochastic processes and their integrals, and derives consequences including mean-variance hedging, existence of equivalent martingale measures, and characterizations of semimartingales. 
Section~\ref{sec:conclusion} concludes.

\section{Neural networks on measurable spaces}\label{sec:networks-measurable-spaces}
\citet[Theorem~1]{hornik91approximation} proved a universal approximation theorem for real-valued functions in $L^p(\mu)$, where $p \in [1, \infty)$ and $\mu$ is a finite Borel measure on $\re^d$, $d\in\na$.
In this section, we extend this result to Orlicz spaces for finite measures on general measurable spaces, including real locally convex Hausdorff topological vector spaces (called locally convex spaces below).
Let $(X, \mathcal{A})$ be a measurable space, and $\Gamma$ a non-empty set of real-valued $\mathcal{A}$-measurable functions.
For example, let $X$ be a complete separable metric space, $\mathcal A = \mathcal B_X$, and let $\Gamma$ be a countable set of real-valued Borel-measurable functions that separates points of $X$.
In that case, $\Gamma$ generates $\mathcal{B}_X$, i.e.\ $\mathcal{B}_X = \sigma(f \colon f \in \Gamma)$ \citep[Theorem~6.8.9]{bogachev07measure2}.
For a locally convex space $X$, the continuous dual $X^\ast$ is a point-separating family of Borel-measurable functions, by the Hahn--Banach theorem \citep[Theorem~3.4]{rudin91functional}.
We retain the general measurable setting; Standing Assumption~\ref{itm:generating-family} will be imposed later for Theorem~\ref{thm:orlicz-approximation}.

\begin{definition}[Neural network]\label{def:nn} 
Let $\psi \from \re \to \re$ be a Borel-measurable function. 
We define $\mathcal{NN}(\Gamma, \psi)$ as the $\re$-linear span of functions of the form 
\begin{equation*}
x \mapsto \psi\bigl(f(x) + \beta\bigr), 
\end{equation*} 
where $f \in \Gamma$ and $\beta \in \re$. 
The elements of $\mathcal{NN}(\Gamma, \psi)$ are called shallow feedforward neural networks with activation function $\psi$. 
\end{definition}

The neural networks in Definition~\ref{def:nn} are real-valued $\mathcal{A}$-measurable functions on $X$.
To extend the definition to Banach-space-valued networks, let $E$ be a real Banach space with a Schauder basis $(e_i)_{i \in \na}$.
Each $z \in E$ has a unique representation $z = \sum_{i=1}^\infty a_i e_i$ with real coefficients $a_i$ and convergence in norm.
Define $\mathcal{NN}(\Gamma, \psi; E)$ to be the $\re$-linear span of the functions of the form
\begin{equation*}
x \mapsto \sum_{i \in I} \psi\bigl(f_i(x) + \beta_i\bigr)e_i,
\end{equation*}
where $I \subset \na$ is finite, $f_i \in \Gamma$ and $\beta_i \in \re$ for all $i \in I$.
Each such network is $\mathcal{A}-\B_E$-measurable and takes values in a finite-dimensional subspace of $E$.
For $E=\re^n$, we use a basis $\{e_1, \ldots, e_n\}$ and take $I \subset \{1, \ldots, n\}$.

\begin{remark}\label{rem:nn-integral}
Every real-valued network $g \in \mathcal{NN}(\Gamma, \psi)$ has a representation of the form
\begin{equation}\label{eq:simple-nn}
g(x) = \sum_{i=1}^n \alpha_i \psi\bigl(f_i(x) + \beta_i\bigr),
\end{equation}
where $n \in \na$ is the number of hidden units in this representation, $\alpha_1, \ldots, \alpha_n$ are real weights, $\beta_1, \ldots, \beta_n$ are bias terms, and $f_1, \ldots, f_n \in \Gamma$.
For the integral representation below, assume that $\psi$ is bounded, that $X$ is a real separable Banach space, and that $\Gamma = X^\ast$.
Define $\Theta \coloneqq \Gamma \times \re$.
Assume also that $\Gamma$ is separable in the operator norm.
Then $\Theta$ is separable in the product topology and $\B_\Gamma \otimes \B_\re = \B_{\Theta}$ \citep[Theorem~15.53]{schmock24stochastic}.
Let $\mu$ be a finite signed Borel measure on $(\Theta, \B_\Gamma \otimes \B_\re)$. 
The representation~\eqref{eq:simple-nn} generalizes to the integral representation
\begin{equation}\label{eq:integral-nn}
g_\mu(x) = \int_\Theta \psi\bigl(f(x) + \beta\bigr)\,\mu(\mathrm{d}f, \mathrm{d}\beta).
\end{equation}
The map $(x,f,\beta)\mapsto\psi(f(x)+\beta)$ is bounded and $\mathcal{A} \otimes \B_\Theta$-measurable, so $g_\mu$ is a bounded $\mathcal{A}$-measurable function.
We recover~\eqref{eq:simple-nn} by choosing $\mu = \sum_{i=1}^n \alpha_i \delta_{z_i}$, with $z_i = (f_i, \beta_i)$.
For $X = \re^d$, the function $x \mapsto \psi(f(x) + \beta)$ is a ridge function; see \citet[Section~5.1]{unser23ridges} for the relation between ridge functions and the Radon transform.
Consider a convex optimization problem for fitting networks of the form~\eqref{eq:integral-nn} to finitely many data points, with a penalty on the total-variation norm of $\mu$.
Under suitable compactness and continuity assumptions, such problems admit minimizers with finite support, of the form $\mu = \sum_{i=1}^n \alpha_i \delta_{z_i}$, where $n \in \na$ and $z_i = (f_i,\beta_i) \in \Theta$ for $i=1,\ldots,n$ \citep{bach17convex}.
These minimizers correspond to finite sums of the form~\eqref{eq:simple-nn}.
\end{remark}

\begin{remark}
The separability assumption on $\Gamma$ in Remark~\ref{rem:nn-integral} need not hold.
A real separable Banach space need not have a separable continuous dual.
Let $\lambda$ be the Lebesgue--Borel measure on $([0,1], \B_{[0,1]})$.
Both $L^1(\lambda)$ and $L^\infty(\lambda)$ are Banach spaces \citep[Theorem~19.1]{billingsley12probability}.
Moreover, $(L^1(\lambda))^\ast \simeq L^\infty(\lambda)$ \citep[Theorem~19.3]{billingsley12probability}.
Since $\B_{[0,1]}$ is countably generated, $L^1(\lambda)$ is separable \citep[Theorem~19.2]{billingsley12probability}. 
However, $L^\infty(\lambda)$ is not separable. 
For $\varepsilon \in (0, 1]$, let $f_\varepsilon = 2\indicatorset{[0, \varepsilon]}$. 
The set $\{ f_\varepsilon \colon \varepsilon \in (0, 1] \}$ contains uncountably many bounded, $\B_{[0,1]}$-measurable functions, and $\| f_\varepsilon - f_{\tilde{\varepsilon}} \|_{L^\infty(\lambda)} = 2$ for all $\varepsilon, \tilde{\varepsilon} \in (0, 1]$ with $\varepsilon \neq \tilde{\varepsilon}$. 
This implies that the open unit balls $(B_1(f_\varepsilon))_{\varepsilon \in (0, 1]}$ are pairwise disjoint.
Since the family is uncountable, $L^\infty(\lambda)$ has no countable dense subset.
\end{remark}

We give several examples of spaces $X$ and families $\Gamma$, together with the corresponding sets $\mathcal{NN}(\Gamma, \psi)$ of neural networks.
In particular, Example~\ref{ex:martingales} treats neural networks on spaces of martingales.
\begin{example}\label{ex:hilbert}
Let $(X, \langle \cdot, \cdot \rangle_X)$ be a real separable Hilbert space.
We equip $X$ with its norm topology and set $\mathcal{A} = \B_X$.
By the Fr\'{e}chet--Riesz representation theorem \citep[Theorem~4.12]{rudin87analysis}, the continuous dual $X^\ast$ is isometrically isomorphic to $X$.
More precisely, for each $F \in X^\ast$, there exists a unique $h \in X$ such that $F(x) = \langle h,x \rangle_X$ for all $x \in X$.
Moreover, $\|F\|_{\mathrm{op}} = \|h\|_X$.
For $\Gamma = X^\ast$, the set $\mathcal{NN}(\Gamma, \psi)$ is the $\re$-linear span of functions of the form
\begin{equation*}
x \mapsto \psi\bigl(\langle h,x \rangle_X + \beta\bigr),
\end{equation*}
where $h \in X$ and $\beta \in \re$.
Example~\ref{ex:deterministic-integrands} below is the special case $X = \Lambda^2$.
\end{example}

\begin{example}\label{ex:deterministic-integrands}
We consider the Hilbert space of deterministic integrands used for importance sampling in \citet{arandjelovic25importance}.
Fix a time horizon $T>0$.
Let $(\Omega, \Fcal, \PP)$ be a probability space endowed with a filtration $\FF = (\Fcal_t)_{t \in [0,T]}$ such that $\Fcal_0$ contains all $\PP$-null sets of $\Fcal$.
Let $d\in\na$, and let $M = (M_t)_{t \in [0,T]}$ be a continuous $\re^d$-valued local $(\FF, \PP)$-martingale with deterministic covariation process $[M]$.
There exist a $\mathcal{B}_{[0,T]}$-measurable function $\pi \from [0,T] \to \re^{d \times d}$ with values in the symmetric positive semidefinite matrices and a finite Lebesgue--Stieltjes measure $\mu$ on $\mathcal{B}_{[0,T]}$ such that $[M]_t = \int_0^t \pi(s)\,\mu(\mathrm{d}s)$ for $t \in [0,T]$ \citep[Remark~2.5]{arandjelovic25importance}.
Let $\Lambda^2$ denote the set of all equivalence classes of $\mathcal{B}_{[0,T]}$-measurable functions $f \from [0,T] \to \re^d$ with $\int_0^T f^\top(s)\pi(s)f(s)\,\mu(\mathrm{d}s) < \infty$. 
We identify $f$ and $g$ if $(f-g)^\top\pi(f-g) = 0$ $\mu$-almost everywhere.

Endowed with the inner product $\langle f,g \rangle = \int_0^T f^\top(s)\pi(s)g(s)\,\mu(\mathrm{d}s)$, $\Lambda^2$ is a real separable Hilbert space, independent of the choice of $(\pi,\mu)$ satisfying the covariation representation \citep[Lemma~2.10(a),(e)]{arandjelovic25importance}.
We equip $X = \Lambda^2$ with its norm topology and Borel $\sigma$-algebra $\mathcal{A} = \B_{\Lambda^2}$.
As in Example~\ref{ex:hilbert}, the continuous dual $(\Lambda^2)^\ast$ is isometrically isomorphic to $\Lambda^2$.
More precisely, for each $F \in (\Lambda^2)^\ast$, there exists a unique $g \in \Lambda^2$ such that $F(f) = \int_0^T g^\top(s)\pi(s)f(s)\,\mu(\mathrm{d}s)$ for all $f \in \Lambda^2$.
Let $\Gamma = (\Lambda^2)^\ast$.
The set $\mathcal{NN}(\Gamma, \psi)$ of neural networks on $\Lambda^2$ is then the $\re$-linear span of functions of the form
\begin{equation*}
f \mapsto \psi\Bigl( \int_0^T g^\top(s)\pi(s)f(s)\,\mu(\mathrm{d}s) + \beta \Bigr),
\end{equation*}
where $g \in \Lambda^2$ and $\beta \in \re$.
\end{example}

\begin{example}\label{ex:continuous-functions}
Let $K$ denote a non-empty compact Hausdorff space, and let $X = C(K)$ be the Banach space of real-valued continuous functions on $K$, endowed with the supremum norm $\| f \| = \sup_{x \in K} |f(x)|$.
Let $\mathcal{M}_{\mathrm{r}}(K)$ denote the space of finite signed Radon measures on $(K,\B_K)$.
The Riesz representation theorem shows that $X^\ast$ can be identified with $\mathcal{M}_{\mathrm{r}}(K)$ \citep[Theorem~6.19]{rudin87analysis}.
More precisely, for each $F \in X^\ast$, there exists a unique $\mu \in \mathcal{M}_{\mathrm{r}}(K)$ such that $F(f) = \int_K f(x)\,\mu(\mathrm{d}x)$ for all $f \in X$.
Let $\Gamma = X^\ast \simeq \mathcal{M}_{\mathrm{r}}(K)$ and $\mathcal{A} = \sigma(\Gamma)$.
The set $\mathcal{NN}(\Gamma, \psi)$ of neural networks on $C(K)$ is then the $\re$-linear span of functions of the form
\begin{equation*}
f \mapsto \psi\Bigl( \int_K f(x)\,\mu(\mathrm{d}x) + \beta \Bigr),
\end{equation*}
where $\mu \in \mathcal{M}_{\mathrm{r}}(K)$ and $\beta \in \re$.
\end{example}

\begin{example}\label{ex:cadlag-functions}
Let $\mathcal{D}$ denote the space of real-valued c\`{a}dl\`{a}g functions on $[0,1]$.
Write $\mathcal{D}_B$ for $\mathcal{D}$ endowed with the supremum norm $\|f\|_\infty = \sup_{t \in [0,1]} |f(t)|$, and $\mathcal{D}_S$ for the same space endowed with the Skorokhod $J_1$ topology.
The space $X = \mathcal{D}_B$ is a Banach space.
Let $\Gamma = X^\ast = \mathcal{D}_B^\ast$ and $\mathcal{A} = \sigma(\Gamma)$.
By \citet[Theorem~3]{pestman95measurability}, $\mathcal{A} = \B_{\mathcal{D}_S}$.

For $f \in \mathcal{D}$ and $t \in (0,1]$, write $\Delta f(t) = f(t)-f(t-)$.
Let $\ell^1((0,1])$ denote the space of functions $h\from (0,1] \to \re$ that vanish outside a countable set and satisfy $\sum_{t \in (0,1]} |h(t)| < \infty$.
By \citet[Theorem~1]{pestman95measurability}, for each $F \in \mathcal{D}_B^\ast$, there exist unique $\mu \in \mathcal{M}_{\mathrm{r}}([0,1])$ and $h \in \ell^1((0,1])$ such that
\begin{equation*}
F(f) = \int_{[0,1]} f(t)\,\mu(\mathrm{d}t) + \sum_{t \in (0,1]} h(t)\Delta f(t),\qquad f \in \mathcal{D}.
\end{equation*}
The series converges absolutely since $|\Delta f(t)| \leq 2\|f\|_\infty$.
Conversely, every such pair defines an element of $\mathcal{D}_B^\ast$.
The set $\mathcal{NN}(\Gamma, \psi)$ of neural networks on $\mathcal{D}$ is then the $\re$-linear span of functions of the form
\begin{equation*}
f \mapsto \psi\Bigl(\int_{[0,1]} f(t)\,\mu(\mathrm{d}t) + \sum_{t \in (0,1]} h(t)\Delta f(t) + \beta \Bigr),
\end{equation*}
where $\mu \in \mathcal{M}_{\mathrm{r}}([0,1])$, $h \in \ell^1((0,1])$, and $\beta \in \re$.
\end{example}

\begin{example}\label{ex:martingales}
Let $(\Omega, \Fcal, \PP)$ be a complete probability space endowed with a right-continuous filtration $\FF = (\Fcal_t)_{t \in \replus}$ such that $\Fcal_0$ contains all $\PP$-null sets of $\Fcal$, and let $\Fcal_\infty = \bigvee_{t \in \replus} \Fcal_t$.
In that case, every $(\FF, \PP)$-martingale has a c\`{a}dl\`{a}g modification \citep[Corollary~5.1.9]{cohen15calculus}, and we shall always implicitly consider such a version when speaking of an $(\FF, \PP)$-martingale.
Moreover, we identify indistinguishable processes.
For a c\`{a}dl\`{a}g process $M = (M_t)_{t \in \replus}$ and $t \in [0, \infty]$, let $M_t^\ast \coloneqq \sup_{s \in [0, t] \cap \replus} |M_s|$.

Given $p \in [1, \infty)$, let $X = \mathcal{H}^p$, the Banach space of all real-valued $(\FF, \PP)$-martingales $M$ with $\| M \|_{\mathcal{H}^p} \coloneqq \|M_\infty^\ast\|_p < \infty$ \citep[Lemma~10.1.5]{cohen15calculus}.
Fix $p \in (1, \infty)$ and let $q = p/(p-1)$.
Then $\| M \|_{\mathcal{H}^p}$ is equivalent to $\| M_\infty \|_p$ on $\mathcal{H}^p$ \citep[Lemma~10.1.3]{cohen15calculus}, where $M_\infty$ denotes the almost sure limit of $M$.
Therefore, $\mathcal{H}^p$ can be identified with the Banach space $L^p(\Omega, \Fcal_\infty, \PP)$ by the map $M \mapsto M_\infty$.
Its inverse associates $Z \in L^p(\Omega, \Fcal_\infty, \PP)$ with the c\`{a}dl\`{a}g version of $(\mathbb{E}[Z \mid \Fcal_t])_{t \in \replus}$.
Moreover, if $\Fcal_\infty$ is countably generated, then $L^p(\Omega, \Fcal_\infty, \PP)$ and thus $\mathcal{H}^p$ are separable \citep[Theorem~19.2]{billingsley12probability}.

Since $(L^p(\Omega, \Fcal_\infty, \PP))^\ast$ is isometrically isomorphic to $L^q(\Omega, \Fcal_\infty, \PP)$ \citep[Theorem~19.3]{billingsley12probability}, it follows that $(\mathcal{H}^p)^\ast$ can be identified with $\mathcal{H}^q$.
More precisely, to each $F \in (\mathcal{H}^p)^\ast$ there corresponds a unique $N \in \mathcal{H}^q$ such that $F(M) = \mathbb{E}[ M_\infty N_\infty ]$ for every $M \in \mathcal{H}^p$.
Let $\Gamma = (\mathcal{H}^p)^\ast \simeq \mathcal{H}^q$ and $\mathcal{A} = \sigma(\Gamma)$.
The set $\mathcal{NN}(\Gamma, \psi)$ of neural networks on $\mathcal{H}^p$ is the $\re$-linear span of functions of the form
\begin{equation*}
M \mapsto \psi\big( \mathbb{E}[ M_\infty N_\infty ] + \beta \big),
\end{equation*}
where $N \in \mathcal{H}^q$ and $\beta \in \re$.
For the duality theory when $p=1$, we refer to \citet[Appendix~A.8]{cohen15calculus}.
\end{example}

\subsection{Universal approximation in Orlicz spaces}\label{subsec:orlicz-approximation}
We study universal approximation in Orlicz hearts $M^\Phi(\mu)$ and Orlicz spaces $L^\Phi(\mu)$.
Recall the measurable space $(X, \mathcal{A})$, and let $\mu \neq 0$ be a finite measure on $\mathcal{A}$.
A function $\Phi \from \re \to [0, \infty]$ is called a Young function if it is convex, even, lower semicontinuous and nontrivial (i.e.\ $\Phi(x) \in \replus$ for some $x > 0$), with $\Phi(0) = 0$ and $\lim_{x \to \infty} \Phi(x) = \infty$.
Given a Young function $\Phi$, its convex (Legendre--Fenchel) conjugate is defined by $\Psi(y) = \sup_{x \ge 0} (x |y| - \Phi(x))$ for every $y \in \re$.
The function $\Psi$ is again a Young function, and $(\Phi, \Psi)$ is called a complementary pair of Young functions.
Every $\replus$-valued Young function $\Phi$ is continuous and admits the integral representation $\Phi(x) = \int_0^{|x|} \varphi(t)\, \mathrm{d}t$, where $\varphi \from \replus \to \replus$ is given by $\varphi(0) = 0$ and $\varphi(t) = \Phi'_-(t)$ for $t > 0$ \citep[Corollary~1.3.2]{rao91orlicz}.
Here $\Phi'_-$ denotes the left derivative; the function $\varphi$ is nondecreasing and left-continuous.
For any nonnegative real-valued random variable $\xi$ on a probability space $(\Omega, \Fcal, \PP)$, Tonelli's theorem gives
\begin{equation*}
\mathbb{E}[\Phi(\xi)] = \int_0^\infty \PP(\xi > t)\varphi(t)\, \mathrm{d}t = \int_0^\infty \PP(\xi \ge t)\varphi(t)\, \mathrm{d}t,
\end{equation*}
where the expectation and integrals may be infinite.
This partially generalizes \citet[Lemma~4.4]{kallenberg21foundations} from $\Phi(x) = |x|^p$ for $p \ge 1$ to all $\replus$-valued Young functions $\Phi$ such as, for example, $\Phi(x) = \mathrm{e}^{x^2}-1$, where $\varphi(t) = 2t\mathrm{e}^{t^2}$.

Given a complementary pair $(\Phi, \Psi)$ of Young functions, let $\mathcal{L}^\Phi(\mu)$ consist of all $\mathcal{A}$-measurable functions $f \from X \to \re$ such that $\int_X \Phi(\alpha f)\, \mathrm{d}\mu < \infty$ for some $\alpha > 0$.
The Orlicz space $L^\Phi(\mu)$ consists of the equivalence classes of functions in $\mathcal{L}^\Phi(\mu)$ under $\mu$-almost everywhere equality.
Let $\mathcal{A}^\Psi$ denote the set of all $g \in L^\Psi(\mu)$ such that $\int_X \Psi(g)\, \mathrm{d}\mu \le 1$.
We define the Orlicz norm $\| \cdot \|_\Phi$ on $L^\Phi(\mu)$ as 
\begin{equation*} 
\| f \|_\Phi = \sup_{g \in \mathcal{A}^\Psi} \int_X |fg|\, \mathrm{d}\mu.
\end{equation*} 
The space $(L^\Phi(\mu), \| \cdot \|_\Phi)$ is a Banach space \citep[Proposition~3.3.11]{rao91orlicz}, and the Orlicz norm is equivalent to the Luxemburg norm $N_\Phi(\cdot)$ \citep[Proposition~3.3.4]{rao91orlicz}.
The set $\mathcal{M}^\Phi(\mu)$ consists of all $\mathcal{A}$-measurable $f \from X \to \re$ such that $\int_X \Phi(\alpha f)\, \mathrm{d}\mu < \infty$ for every $\alpha > 0$, and the Orlicz heart $M^\Phi(\mu)$ consists of the corresponding equivalence classes under $\mu$-almost everywhere equality and is endowed with the restriction of $\| \cdot \|_\Phi$.

A Young function $\Phi$ is said to satisfy the $\Delta_2$-condition ($\Phi \in \Delta_2$) if $\Phi(\re) \subset \replus$ and there exist constants $x_0 > 0$ and $K > 0$ such that $\Phi(2x) \le K \Phi(x)$ for all $x \ge x_0$.
In that case, $M^\Phi(\mu) = L^\Phi(\mu)$, and simple functions are dense in $L^\Phi(\mu)$ \citep[Corollary~3.4.5]{rao91orlicz}. 
In particular, if $\Phi \in \Delta_2$, then $(M^\Phi(\mu), \| \cdot \|_\Phi)$ is a Banach space.
Examples of Young functions satisfying or failing the $\Delta_2$-condition are: 
\begin{itemize}
\item For each $p \in [1, \infty)$ and $a > 0$, the function $\Phi(x) = a|x|^p$ satisfies the $\Delta_2$-condition.
\item The function $\Phi(x) = \mathrm{e}^{|x|} - 1$ does not satisfy the $\Delta_2$-condition.
\item The function $\Phi(x) = x^2 / \log(\mathrm{e} + |x|)$ satisfies the $\Delta_2$-condition.
\item The function $\Phi(x) = (1+|x|)\log(1+|x|) - |x|$ satisfies the $\Delta_2$-condition.
\end{itemize}

\begin{lemma}\label{lem:orlicz-l1}
For every Young function $\Phi$, we have $L^\Phi(\mu) \subset L^1(\mu)$.
\end{lemma}

\begin{proof}
Since the convex conjugate $\Psi$ is a Young function, there exists $a > 0$ such that $b \coloneqq \Psi(a) < \infty$.
Then $b \ge 0$, and the definition of $\Psi$ gives $\Phi(x) \ge ax - b$ for every $x \in \replus$.
For $f \in L^\Phi(\mu)$, let $c > 0$ be such that $\int_X \Phi(c f)\, \mathrm{d}\mu < \infty$.
Since $\Phi$ is even, $ac|f| \le \Phi(c|f|) + b = \Phi(cf) + b$, hence
\begin{equation*}
\| f \|_1 \le \frac{1}{ac} \Big( \int_X \Phi(c f)\, \mathrm{d}\mu + b \mu(X) \Big),
\end{equation*}
where the right-hand side is finite.
\end{proof}

The following criterion will be used to prove norm convergence of conditional expectations.
\begin{lemma}\label{lem:orlicz-mean-convergence}
Let $\Phi \in \Delta_2$, let $(f_n)_{n \in \na}$ be a sequence in $L^\Phi(\mu)$, and let $f \in L^\Phi(\mu)$.
Then $\| f_n-f \|_\Phi \to 0$ if and only if $f_n \to f$ in $\mu$-measure and $\int_X \Phi(f_n-f)\, \mathrm{d}\mu \to 0$.
The latter condition is called mean convergence.
In particular, $\| f_n-f \|_\Phi \to 0$ if $f_n \to f$ in $\mu$-measure and $\int_X \Phi(f_n)\, \mathrm{d}\mu \to \int_X \Phi(f)\, \mathrm{d}\mu$.
\end{lemma}

\begin{proof}
Put $h_n=f_n-f$.
By equivalence of the Orlicz and Luxemburg norms, we may work with $N_\Phi(h)=\inf\{a>0:\int_X\Phi(h/a)\,\mathrm{d}\mu\le1\}$.
If $N_\Phi(h_n)\to0$, choose $a_n>N_\Phi(h_n)$ with $a_n\to0$.
Then $\int_X\Phi(h_n/a_n)\,\mathrm{d}\mu\le1$.
Convexity of $\Phi$ implies $\int_X\Phi(h_n)\,\mathrm{d}\mu\le a_n$ for all sufficiently large $n$ and thus $\int_X\Phi(h_n)\, \mathrm d \mu \to 0$.
For every $\delta>0$, Chebyshev's inequality gives, for all sufficiently large $n$,
\begin{equation*}
\mu(|h_n|>\delta)\le\frac{1}{\Phi(\delta/a_n)}\to0,
\end{equation*}
and therefore $f_n \to f$ in $\mu$-measure.

Conversely, suppose that $h_n\to0$ in $\mu$-measure and $\int_X\Phi(h_n)\,\mathrm{d}\mu\to0$.
Fix $c>0$.
By monotonicity of $\Phi$ and iteration of the $\Delta_2$-inequality, there exist $R>0$ and $C>0$ such that $\Phi(ct)\le C\Phi(t)$ for $|t|\ge R$.
For $0<\delta<R$, splitting the integral at $\delta$ and $R$ gives
\begin{equation*}
\int_X\Phi(ch_n)\,\mathrm{d}\mu\le \mu(X)\Phi(c\delta)+\Phi(cR)\mu(|h_n|>\delta)+C\int_X\Phi(h_n)\,\mathrm{d}\mu.
\end{equation*}
Letting first $n\to\infty$ and then $\delta\downarrow0$ yields $\int_X\Phi(ch_n)\,\mathrm{d}\mu\to0$.
Since $c>0$ was arbitrary, the definition of $N_\Phi$ gives $N_\Phi(h_n)\to0$.

For the last assertion, the $\Delta_2$-condition and continuity of $\Phi$ give constants $C>0$ and $D\ge0$ such that $\Phi(2t)\le C\Phi(t)+D$ for every $t\in\re$.
By convexity,
\begin{equation*}
0\le\Phi(h_n)\le \tfrac12\big(\Phi(2f_n)+\Phi(2f)\big)\le \tfrac C2\big(\Phi(f_n)+\Phi(f)\big)+D\eqqcolon w_n.
\end{equation*}
All these functions are integrable since $L^\Phi(\mu)=M^\Phi(\mu)$.
Moreover, as $f_n \to f$ in $\mu$-measure, every subsequence has a further subsequence along which $f_n\to f$ almost everywhere.
Along such a subsequence, retaining the index $n$, we have $w_n\to w\coloneqq C\Phi(f)+D$ almost everywhere and $\int_Xw_n\,\mathrm{d}\mu\to\int_Xw\,\mathrm{d}\mu$.
Fatou's lemma applied to $w_n-\Phi(h_n)\ge0$ gives
\begin{equation*}
\int_Xw\,\mathrm{d}\mu\le\liminf_{n\to\infty}\int_X\big(w_n-\Phi(h_n)\big)\,\mathrm{d}\mu=\int_Xw\,\mathrm{d}\mu-\limsup_{n\to\infty}\int_X\Phi(h_n)\,\mathrm{d}\mu,
\end{equation*}
and thus $\limsup_{n\to\infty}\int_X\Phi(h_n)\,\mathrm{d}\mu\leq 0$.
Since $\Phi(h_n)\geq 0$, we have $\int_X\Phi(h_n)\,\mathrm{d}\mu\to0$ along this subsequence, and hence along the original sequence.
\end{proof}

\begin{remark}
Convergence in $\mu$-measure cannot in general be omitted.
Indeed, for the Young function $\Phi(x)=(|x|-1)_+$ and $(X,\mathcal{A},\mu) = (\{0\},2^{\{0\}},\delta_0)$, taking the constant sequence $f_n\equiv1$ and $f\equiv0$ yields $\int_X \Phi(f_n-f)\,\mathrm{d}\mu=0$ for all $n$, whereas $\|f_n-f\|_\Phi=1$ for all $n$.
\end{remark}

Recall that for $f \in L^2(\mu)$ and a sub-$\sigma$-algebra $\tilde{\mathcal{A}}\subset\mathcal{A}$, the conditional expectation $\mathbb{E}_\mu[f\, |\, \tilde{\mathcal{A}}]$ is the orthogonal projection in $L^2(\mu)$ of $f$ onto the closed subspace of all $g \in L^2(\mu)$ that are $\tilde{\mathcal{A}}$-measurable \citep[Chapter~27]{schilling17measures}.
This operator extends to $L^p(\mu)$ for every $p \in [1,\infty]$, mapping $L^p(\mu)$ onto the subspace of all $g \in L^p(\mu)$ that are $\tilde{\mathcal{A}}$-measurable \citep[Theorem~27.5]{schilling17measures}. 
In particular, Lemma~\ref{lem:orlicz-l1} implies that $\mathbb{E}_\mu[f\, |\, \tilde{\mathcal{A}}]$ is well defined for every Young function $\Phi$ and every $f \in L^\Phi(\mu)$.
Using Lemma~\ref{lem:orlicz-mean-convergence}, we obtain the following convergence theorem in Orlicz spaces.
For convergence results in Orlicz spaces involving nets of conditioning $\sigma$-rings, see \citet{krickeberg64conditional}.
\begin{theorem}\label{thm:orlicz-conditional-expectations}
Let $(\mathcal{A}_n)_{n \in \na}$ be an increasing sequence of sub-$\sigma$-algebras of $\mathcal{A}$, and let $\mathcal{A}_\infty \coloneqq \bigvee_{n \in \na} \mathcal{A}_n$.
If $\Phi \in \Delta_2$, then for every $g \in L^\Phi(\mu)$, we have
\begin{equation*}
\mathbb{E}_\mu[g\, |\, \mathcal{A}_n] \to \mathbb{E}_\mu[g\, |\, \mathcal{A}_\infty], \quad n \to \infty,
\end{equation*}
where convergence holds $\mu$-almost everywhere and in $L^\Phi(\mu)$.
\end{theorem}

\begin{proof}
Let $M_n = \mathbb{E}_\mu[g\, |\, \mathcal{A}_n]$ for $n \in \na \cup \{\infty\}$.
Since $\Phi \in \Delta_2$, we have $L^\Phi(\mu)=M^\Phi(\mu)$ and hence $\Phi(g)\in L^1(\mu)$.
By conditional Jensen's inequality \citep[Theorem~27.16]{schilling17measures}, for every $n \in \na \cup \{\infty\}$,
\begin{equation*}
\int_X \Phi(M_n)\,\mathrm{d}\mu\le \int_X \mathbb{E}_\mu[\Phi(g)\, |\, \mathcal{A}_n]\,\mathrm{d}\mu= \int_X \Phi(g)\,\mathrm{d}\mu < \infty.
\end{equation*}
Thus $M_n\in L^\Phi(\mu)$ for every $n \in \na \cup \{\infty\}$.

The convergence theorem for conditional expectations \citep[Theorem~27.19(i)]{schilling17measures} gives $M_n\to M_\infty$ both $\mu$-almost everywhere and in $L^1(\mu)$.
By the tower property, $M_n=\mathbb{E}_\mu[M_\infty\, |\, \mathcal{A}_n]$ for every $n\in\na$.
Another application of conditional Jensen's inequality gives
\begin{equation*}
\int_X \Phi(M_n)\,\mathrm{d}\mu\le \int_X \mathbb{E}_\mu[\Phi(M_\infty)\, |\, \mathcal{A}_n]\,\mathrm{d}\mu = \int_X \Phi(M_\infty)\,\mathrm{d}\mu.
\end{equation*}
Since $\Phi$ is $\replus$-valued and convex, it is continuous, hence $\Phi(M_n)\to\Phi(M_\infty)$, $\mu$-almost everywhere.
Fatou's lemma now yields
\begin{equation*}
\int_X \Phi(M_\infty)\,\mathrm{d}\mu\le \liminf_{n\to\infty}\int_X \Phi(M_n)\,\mathrm{d}\mu \le \limsup_{n\to\infty}\int_X \Phi(M_n)\,\mathrm{d}\mu\le \int_X \Phi(M_\infty)\,\mathrm{d}\mu.
\end{equation*}
Hence $\int_X\Phi(M_n)\,\mathrm{d}\mu\to\int_X\Phi(M_\infty)\,\mathrm{d}\mu$.
Since $\mu$ is finite, almost-everywhere convergence also gives convergence in $\mu$-measure.
Norm convergence follows from the last assertion of Lemma~\ref{lem:orlicz-mean-convergence}.
\end{proof}

We recall the dual representation of Orlicz hearts; see \citet[Theorems~4.1.6--4.1.7 and Corollary~4.1.9]{rao91orlicz}.
\begin{theorem}\label{thm:orlicz-duality}
Let $(\Phi,\Psi)$ be a complementary pair of Young functions, and assume that $\Phi$ is $\replus$-valued and $\Phi(x)=0$ if and only if $x=0$.
To each $F\in(M^\Phi(\mu))^\ast$ there corresponds a unique $g\in L^\Psi(\mu)$ such that
\begin{equation*}
F(f)=\int_X fg\,\mathrm{d}\mu, \qquad f\in M^\Phi(\mu).
\end{equation*}
Conversely, each $g\in L^\Psi(\mu)$ defines a continuous linear functional by this formula.
Moreover, $\|F\|_{\mathrm{op}}=N_\Psi(g)$.
Thus, $(M^\Phi(\mu))^\ast$ is isometrically isomorphic to $(L^\Psi(\mu),N_\Psi)$. 
If $\Phi\in\Delta_2$, the same representation and norm identity hold with $L^\Phi(\mu)$ in place of $M^\Phi(\mu)$.
\end{theorem}

\begin{assumption}\label{ass:space-activation}
For the remainder of this section, assume that
\begin{enumerate}[label=(\alph*),ref=\ref{ass:space-activation}(\alph*)]
\item \label{itm:generating-family} There is a non-empty family $\Gamma_0\subseteq\Gamma$, closed under pointwise addition and subtraction, such that $\mathcal{A}=\sigma(\Gamma_0)$.
\item \label{itm:sigmoidal-activation} The activation function $\psi \from \re \to \re$ is bounded, Borel-measurable and sigmoidal, meaning that $\lim_{x\to-\infty}\psi(x)=0$ and $\lim_{x\to+\infty}\psi(x)=1$.
\end{enumerate}
\end{assumption}
Assumption~\ref{itm:generating-family} implies $\mathcal{A}=\sigma(\Gamma)$ and includes the choice $\Gamma_0=\Gamma=X^\ast$ and $\mathcal{A}=\sigma(X^\ast)$ on a real locally convex Hausdorff topological vector space.
Neither a topology on $X$ nor separation of its points by $\Gamma$ is required.
The family $\Gamma_0$ need not be closed under multiplication by arbitrary real scalars.
In Assumption~\ref{itm:sigmoidal-activation}, we follow \citet{cybenko89approximation}.

The next proposition extends the discrimination argument of \citet{cybenko89approximation} to the measurable setting of Standing Assumption~\ref{ass:space-activation}.
\begin{proposition}\label{prop:discriminatory-activation}
The activation function $\psi$ is discriminatory: if a finite signed measure $\nu$ on $\mathcal{A}$ satisfies $\int_X\psi(f(x)+\beta)\,\nu(\mathrm{d}x)=0$ for all $f\in\Gamma$ and $\beta\in\re$, then $\nu=0$.
\end{proposition}

\begin{proof}
Let $\nu$ be a finite signed measure on $\mathcal A$ that satisfies $\int_X\psi(f(x)+\beta)\,\nu(\mathrm{d}x)=0$ for all $f\in\Gamma$ and $\beta\in\re$.
Since $0\in\Gamma_0$ and $\psi$ is not identically zero, taking $f=0$ gives $\nu(X)=0$.
Fix $f\in\Gamma_0$ and $\theta,\gamma\in\re$.
For every $n\in\na$, we have $nf\in\Gamma_0\subseteq\Gamma$.
Taking $\beta=\gamma-n\theta$ and applying dominated convergence with respect to $|\nu|$ gives
\begin{equation*}
0=\lim_{n\to\infty}\int_X\psi\bigl(n(f(x)-\theta)+\gamma\bigr)\,\nu(\mathrm{d}x)=\nu\{f>\theta\}+\psi(\gamma)\nu\{f=\theta\}.
\end{equation*}
Letting $\gamma\to-\infty$ yields $\nu\{f>\theta\}=0$.
Let $\eta\coloneqq\nu\circ f^{-1}$.
Then $\eta((\theta,\infty))=0$ for every $\theta\in\re$.
The upper half-lines form an intersection-stable family generating $\B_{\re}$.
Since $\eta(\re)=\nu(X)=0$, by
\citet[Theorem~1.9.3(ii)]{bogachev07measure1}, we have $\eta=0$.
Consequently,
\begin{equation}\label{eq:vanishing-fourier-transform}
\int_X\exp(\mathrm{i}t f(x))\,\nu(\mathrm{d}x)=0,\qquad f\in\Gamma_0,\quad t\in\re.
\end{equation}

Fix $m\in\na$ and $f_1,\ldots,f_m\in\Gamma_0$.
Write $F=(f_1,\ldots,f_m)$ and let $\eta=\nu\circ F^{-1}$.
For $q\in\mathbb{Q}^m$, choose $N\in\na$ and $k_1,\ldots,k_m\in\mathbb{Z}$ such that $q_j=k_j/N$.
Then $h=\sum_{j=1}^m k_jf_j\in\Gamma_0$, so \eqref{eq:vanishing-fourier-transform} implies
\begin{equation*}
\widehat{\eta}(q) =\int_X\exp\Bigl(\mathrm{i}\sum_{j=1}^m q_jf_j(x)\Bigr) \,\nu(\mathrm{d}x) =\int_X\exp\bigl(\mathrm{i}h(x)/N\bigr)\,\nu(\mathrm{d}x)=0.
\end{equation*}
By dominated convergence, $\widehat{\eta}$ is continuous, hence it vanishes on $\re^m$.
The characteristic functions of the positive and negative parts $\eta^+$ and $\eta^-$ therefore agree.
By \citet[Lemma~7.13.5]{bogachev07measure2}, we have $\eta^+=\eta^-$, and thus $\eta=0$.
It follows that $\nu$ vanishes on every set $F^{-1}(B)$ with $B\in\B_{\re^m}$.
These finite-dimensional cylinder sets form an intersection-stable family generating $\sigma(\Gamma_0)=\mathcal{A}$.
Hence, by \citet[Theorem~1.9.3(ii)]{bogachev07measure1}, we have $\nu=0$.
\end{proof}

The following theorem extends Theorem~1 of \citet{hornik91approximation} in two directions.
First, we pass from $L^p$-spaces to Orlicz hearts and, under the $\Delta_2$-condition, to Orlicz spaces.
Second, we replace the Euclidean input space by the general measurable setting introduced above.
Recall that $\mu \neq 0$ is a finite measure on $\mathcal{A} = \sigma(\Gamma_0) = \sigma(\Gamma)$, and let $\Psi$ be the Young function complementary to $\Phi$.

\begin{theorem}\label{thm:orlicz-approximation}
Let $\Phi$ be an $\replus$-valued Young function.
Then the set $\mathcal{NN}(\Gamma, \psi)$ of neural networks is dense in the Orlicz heart $M^\Phi(\mu)$ with respect to $\|\cdot\|_\Phi$.
Moreover, if $\Phi \in \Delta_2$, then $\mathcal{NN}(\Gamma, \psi)$ is dense in the Orlicz space $L^\Phi(\mu)$.
\end{theorem}

\begin{proof}
Since every Young function is even and nondecreasing on $\replus$, for every
$f \in L^\infty(\mu)$ and $\alpha > 0$,
\begin{equation*}
\int_X \Phi(\alpha f)\,\mathrm{d}\mu
\le \Phi(\alpha\|f\|_\infty)\mu(X) < \infty.
\end{equation*}
Thus $L^\infty(\mu) \subset M^\Phi(\mu)$.
Since $\psi$ is bounded, $\mathcal{NN}(\Gamma,\psi) \subset L^\infty(\mu)$ and therefore $\mathcal{NN}(\Gamma,\psi) \subset M^\Phi(\mu)$.
If $\mathcal{NN}(\Gamma,\psi)$ were not dense in $M^\Phi(\mu)$, the Hahn--Banach theorem \citep[Theorem~5.19]{rudin87analysis} would yield a nonzero $F \in (M^\Phi(\mu))^\ast$ that vanishes on $\mathcal{NN}(\Gamma,\psi)$.

To apply Theorem~\ref{thm:orlicz-duality}, set $\widetilde{\Phi}(t) \coloneqq \Phi(t)+|t|$ for $t \in \re$, and let $\widetilde{\Psi}$ be its complementary Young function.
By Lemma~\ref{lem:orlicz-l1}, every element of $M^\Phi(\mu)$ is integrable, hence $M^{\smash[t]{\widetilde{\Phi}}}(\mu)=M^\Phi(\mu)$.
Since $\widetilde{\Psi} \le \Psi$, the definition of the Orlicz norm gives $\|h\|_\Phi \le \|h\|_{\widetilde{\Phi}}$ for every $h \in M^\Phi(\mu)$.
Since $F$ is linear and continuous with respect to $\|\cdot\|_\Phi$, it is bounded, so there exists $C>0$ such that $|F(h)| \le C\|h\|_\Phi$ for all $h \in M^\Phi(\mu)$.
Hence $|F(h)| \le C\|h\|_{\widetilde{\Phi}}$ for all $h \in M^{\widetilde{\Phi}}(\mu)$, and therefore $F$ is continuous with respect to $\|\cdot\|_{\widetilde{\Phi}}$, which implies $F \in (M^{\smash[t]{\widetilde{\Phi}}}(\mu))^\ast$.
The function $\widetilde{\Phi}$ is $\replus$-valued and vanishes only at zero, so Theorem~\ref{thm:orlicz-duality} yields $g \in L^{\smash[t]{\widetilde{\Psi}}}(\mu)$ such that
\begin{equation*}
F(h)=\int_X hg\,\mathrm{d}\mu,
\qquad h \in M^\Phi(\mu).
\end{equation*}
By Lemma~\ref{lem:orlicz-l1}, we have $g \in L^1(\mu)$.
Consequently, $\nu(A) \coloneqq \int_A g\,\mathrm{d}\mu$, $A \in \mathcal{A}$, defines a finite signed measure with $|\nu|(X)=\int_X |g|\,\mathrm{d}\mu<\infty$.
Since $F$ vanishes on $\mathcal{NN}(\Gamma,\psi)$,
\begin{equation*}
\int_X \psi\bigl(f(x)+\beta\bigr)\,\nu(\mathrm{d}x)=0,
\quad f \in \Gamma,\,\beta \in \re.
\end{equation*}
Proposition~\ref{prop:discriminatory-activation} implies that $\nu=0$, hence $F=0$, a contradiction.

Finally, if $\Phi \in \Delta_2$, then $L^\Phi(\mu)=M^\Phi(\mu)$, and $\mathcal{NN}(\Gamma,\psi)$ is dense in $L^\Phi(\mu)$ as well.
\end{proof}

\section{Neural network approximation of measurable random variables}\label{sec:random-variable-approximation}
For the remainder of this paper, let $(\Omega, \Fcal, \PP)$ denote a complete probability space.
Let $Y = (Y_t)_{t \ge 0}$ be an $\re^{d}$-valued stochastic process for some dimension $d \in \na$.
In analogy to \citet{buehler19hedging}, the process $Y$ represents the observed market information.
We denote by $\FF^0 = (\Fcal_t^0)_{t \ge 0}$ the natural filtration generated by $Y$, i.e.\ $\Fcal_t^0 = \sigma(Y_s \colon 0 \le s \le t)$ for $t \in \replus$, and set $\Fcal_\infty^0 = \sigma(\bigcup_{t \ge 0} \Fcal_t^0) = \sigma(Y_s \colon s \in \replus)$.
As it is common in the literature to work under the usual hypotheses, we therefore consider the augmentation of $\FF^0$ by the $\PP$-null sets.
To this end, let $\mathcal{N}_\PP$ denote the set of all $\PP$-null sets of $\Fcal$.
Let $\FF = (\Fcal_t)_{t \ge 0}$, where $\Fcal_t \coloneqq \sigma(\mathcal{N}_\PP \cup \Fcal_t^0)$ for every $t \in \replus$, and $\Fcal_\infty = \sigma(\mathcal{N}_\PP \cup \Fcal_\infty^0)$.
Then $\FF$ is an augmentation of $\FF^0$, and $\Fcal_0$ contains all $\PP$-null sets of $\Fcal$.
This augmentation need not be right-continuous.

The approximation result of Theorem~\ref{thm:orlicz-approximation} in the preceding section is formulated in terms of abstract random variables.
However, the general framework of Section~\ref{subsec:orlicz-approximation} allows us to study the case when the filtration is generated by an observed stochastic process.
The following theorem shows that, for continuous and for c\`{a}dl\`{a}g processes, measurable random variables can be approximated by neural networks that process the observed path of $Y$.
When $Y$ is a semimartingale, integrating by parts yields an approximation in terms of stochastic integrals of $Y$.

\begin{theorem}\label{thm:path-approximation}
Let $T>0$, suppose that $Y$ has c\`{a}dl\`{a}g paths on $[0,T]$, and let $\psi$ satisfy Assumption~\ref{itm:sigmoidal-activation}.
Let $\Phi$ be an $\replus$-valued Young function and $g\in M^\Phi(\Omega,\Fcal_T,\PP)$.
For $s\in(0,T]$, write $\Delta Y_s=Y_s-Y_{s-}$, and define $\ell^1((0,T])$ as in Example~\ref{ex:cadlag-functions}, with $1$ replaced by $T$.
Then the following assertions hold.

\begin{enumerate}[label=(\alph*)]
\item For every $\varepsilon>0$, there exist $m\in\na$, constants $\gamma_i,b_i\in\re$, finite signed Radon measures $\mu_i^j$ on $[0,T]$, and $h_i^j\in\ell^1((0,T])$ such that
\begin{equation}\label{eq:path-net}
G\coloneqq\sum_{i=1}^m\gamma_i\psi\biggl(\sum_{j=1}^d\int_{[0,T]}Y_s^j\,\mu_i^j(\mathrm{d}s)+J_i+b_i\biggr),\qquad J_i\coloneqq\sum_{j=1}^d\sum_{s\in(0,T]}h_i^j(s)\Delta Y_s^j,
\end{equation}
satisfies $\|G-g\|_\Phi<\varepsilon$.
For continuous $Y$, the jump terms vanish.
If $\Phi\in\Delta_2$, these conclusions hold for every $g\in L^\Phi(\Omega,\Fcal_T,\PP)$.
\item Suppose that $(\Fcal_t)_{0\le t\le T}$ is right-continuous and that $Y$ is an $\FF$-semimartingale on $[0,T]$.
Define the deterministic c\`{a}dl\`{a}g finite-variation processes
$A_t^i\coloneqq-(\mu_i^j([0,t]))_{j=1}^d$ and the $\sigma(Y_T)$-measurable variables $\beta_i\coloneqq b_i-(A_T^i)^\top Y_T$.
Every $G$ in \eqref{eq:path-net} then satisfies
\begin{equation*}
G=\sum_{i=1}^m\gamma_i\psi\biggl(\int_{(0,T]}(A_{s-}^i)^\top\,\mathrm{d}Y_s+J_i+\beta_i\biggr)\quad\PP\text{-almost surely}.
\end{equation*}
\end{enumerate}
\end{theorem}

\begin{proof}
Examples~\ref{ex:continuous-functions} and~\ref{ex:cadlag-functions}, applied coordinatewise and after rescaling time, give precisely the linear functionals appearing inside the activations in part~\textup{(a)}, without the biases.
Set $E=C([0,T];\re^d)$ if $Y$ is continuous and $E=D([0,T];\re^d)$ otherwise, and use the supremum norm to define $E^\ast$ in both cases.
Write $\B_E$ for the uniform-norm Borel $\sigma$-algebra in the first case and the Skorokhod $J_1$ Borel $\sigma$-algebra in the second.
An indistinguishable modification of $Y$ leaves its completed natural filtration unchanged, so we may assume that $Y(\omega) \in E$ for all $\omega \in \Omega$.
For $t\in[0,T]$ and $1\le j\le d$, let $\pi_t^j(x)=x_t^j$ denote the coordinate projections at time $t$.

Suppose first that $E=C([0,T];\re^d)$.
By the finite-dimensional cylinder characterization of the Borel $\sigma$-algebra on $E$ \citep[Problem~2.4.2]{karatzas98calculus}, $\B_E = \sigma\bigl(\pi_t^j \colon t \in [0,T],\ 1 \le j \le d\bigr) = \sigma(E^\ast)$.
Indeed, $\pi_t^j \in E^\ast$ for every $t$ and $j$, whereas every $\ell \in E^\ast$ is continuous and therefore $\B_E$-measurable, which proves the second equality.

Suppose next that $E=D([0,T];\re^d)$.
Here $E^\ast$ denotes the dual with respect to the supremum norm and $\B_E$ the Skorokhod $J_1$ Borel $\sigma$-algebra.
By \citet[Theorem~1]{pestman95measurability}, applied after rescaling $[0,T]$ to $[0,1]$, every $\ell\in E^\ast$ is of the form
\begin{equation*}
\ell(x)=\sum_{j=1}^d\left(\int_{[0,T]}x_s^j\,\mu^j(\mathrm{d}s)+\sum_{s\in(0,T]}h^j(s)\Delta x_s^j\right),
\end{equation*}
where $\mu^j$ is a finite signed Radon measure and $h^j\in\ell^1((0,T])$.
Since $|\Delta x_s^j|\le 2\|x\|_\infty$, each jump series is absolutely convergent.
By \citet[Chapter~3, Proposition~7.1]{ethier86markov}, the Borel $\sigma$-algebra associated with the Skorokhod $J_1$ topology is generated by the coordinate maps, i.e.\ $\B_E=\sigma(\pi_t^j\colon t\in [0,T],\ 1\le j\le d)$.
Moreover, \citet[Theorem~3]{pestman95measurability} yields $\B_E=\sigma(E^\ast)$.
Consequently, in either case,
\begin{equation}\label{eq:path-borel-sigma-algebra}
\sigma(E^\ast)=\sigma\bigl(\pi_t^j\colon t\in [0,T],\ 1\le j\le d\bigr)=\B_E.
\end{equation}

Define the path map $\mathbf{Y}\from\Omega\to E$ by $\mathbf{Y}(\omega)\coloneqq(Y_t(\omega))_{0\le t\le T}$.
By \eqref{eq:path-borel-sigma-algebra},
\begin{equation}\label{eq:path-information}
\sigma(\mathbf{Y})=\sigma\bigl(Y_t^j\colon t\in [0,T],\ 1\le j\le d\bigr)=\sigma\bigl(Y_t\colon 0\le t\le T\bigr)=\Fcal_T^0.
\end{equation}
Since $\Fcal_T=\sigma(\mathcal N_\PP\cup\Fcal_T^0)$, \citet[Lemma~1.27, p.~24]{kallenberg21foundations} yields an $\Fcal_T^0$-measurable random variable $\widetilde g$ such that $\widetilde g=g$ $\PP$-almost surely.
By \eqref{eq:path-information} and the functional representation lemma \citep[Lemma~1.14, p.~18]{kallenberg21foundations}, there exists a $\B_E$-measurable function $f\from E\to\re$ such that $\widetilde g=f(\mathbf{Y})$ and hence $g=f(\mathbf{Y})$ $\PP$-almost surely.
Writing $\nu\coloneqq\PP\circ\mathbf{Y}^{-1}$, we obtain, for every $\alpha>0$,
\begin{equation*}
\int_E\Phi\bigl(\alpha f(x)\bigr)\,\nu(\mathrm dx)=\mathbb{E}\bigl[\Phi(\alpha\widetilde g)\bigr]=\mathbb{E}\bigl[\Phi(\alpha g)\bigr]<\infty,
\end{equation*}
and therefore $f\in M^\Phi(E,\B_E,\nu)$.

By \eqref{eq:path-borel-sigma-algebra}, the family $\Gamma_0=\Gamma=E^\ast$ is closed under addition and subtraction and generates $\B_E$.
Theorem~\ref{thm:orlicz-approximation} therefore yields $\widetilde f\in\mathcal{NN}(E^\ast,\psi)$ such that $\|\widetilde f-f\|_\Phi<\varepsilon/2$ under $\nu$.
The representations of $E^\ast$ above then give \eqref{eq:path-net}, and the assertion follows by setting $G=\widetilde f(\mathbf{Y})$.
If $\Phi\in\Delta_2$, then $L^\Phi=M^\Phi$, and the same argument applies to every $g\in L^\Phi(\Omega,\Fcal_T,\PP)$.

To prove part~\textup{(b)}, fix one measure $\mu=\mu_i^j$ and put $B_t=\mu([0,t])$.
Then $B$ is deterministic, c\`{a}dl\`{a}g and of finite variation, with $B_0=\mu(\{0\})$ and $\Delta B_s=\mu(\{s\})$ for $s>0$.
The product formula for $B$ and the semimartingale $Y^j$ gives
\begin{align*}
B_TY_T^j-B_0Y_0^j
&=\int_{(0,T]}B_{s-}\,\mathrm{d}Y_s^j+\int_{(0,T]}Y_{s-}^j\,\mu(\mathrm{d}s)+\sum_{s\in(0,T]}\mu(\{s\})\Delta Y_s^j\\
&=\int_{(0,T]}B_{s-}\,\mathrm{d}Y_s^j+\int_{(0,T]}Y_s^j\,\mu(\mathrm{d}s);
\end{align*}
see \citet[Chapter~20, Eq.~(1) and Theorem~20.6(viii)]{kallenberg21foundations}.
Since a c\`{a}dl\`{a}g path has at most countably many jumps, the two integrals with respect to $\mu$ differ by precisely the displayed jump sum.
Adding the atom at zero yields
\begin{equation*}
\int_{[0,T]}Y_s^j\,\mu(\mathrm{d}s)=B_TY_T^j-\int_{(0,T]}B_{s-}\,\mathrm{d}Y_s^j.
\end{equation*}
Summing the identity over $j$ and adding $J_i+b_i$ proves part~\textup{(b)}.
\end{proof}

\begin{remark}
The attainability of contingent claims by trading in the available assets is a basic issue in mathematical finance.
In Brownian market models, the martingale representation theorem is the classical starting point.
Suppose that $Y=W$ is a standard $d$-dimensional Brownian motion and that $g\in L^2(\Omega,\Fcal_T,\PP)$.
The martingale $V_t=\mathbb{E}[g\mid\Fcal_t]$ has a continuous version of the form
\begin{equation*}
V_t=\mathbb{E}[g]+\int_0^t H_s^\top\,\mathrm{d}W_s,\qquad 0\le t\le T,
\end{equation*}
for a predictable $H$ with $\mathbb{E}\int_0^T|H_s|^2\,\mathrm{d}s<\infty$.

Such a representation need not hold for a general martingale, even in its own completed natural filtration.
Let $Z$ take the values $-1,0,1$ with equal probabilities, fix $t_0\in(0,T)$, and set $Y_t=Z\indicatorset{[t_0,\infty)}(t)$ with its completed natural filtration $\FF$.
For $s<t_0\le t$,
\begin{equation*}
\mathbb{E}[Y_t\mid\Fcal_s]=\mathbb{E}[Z]=0=Y_s,
\end{equation*}
while the martingale property is immediate in the remaining cases.
If $H$ is predictable, $H_{t_0}$ is $\Fcal_{t_0-}$-measurable.
But $\Fcal_{t_0-}$ is $\PP$-trivial, because $Y_s=0$ for every $s<t_0$.
Hence $H_{t_0}=a$ almost surely for some $a\in\re$.
Since $Y$ has only the jump $\Delta Y_{t_0}=Z$,
\begin{equation*}
\int_{(0,T]} H_s\,\mathrm{d}Y_s=H_{t_0}\Delta Y_{t_0}=aZ.
\end{equation*}
Thus, for $g=Z^2$, no representation $g=\mathbb{E}[g]+\int_{(0,T]} H_s\,\mathrm{d}Y_s$ is possible, since on $\{Z=0\}$ its right-hand side equals $\mathbb{E}[Z^2]=2/3$, whereas its left-hand side equals $0$.

More generally, if $\FF$ is right-continuous and $Y$ is a square-integrable martingale on $[0,T]$, the Galtchouk--Kunita--Watanabe decomposition gives, for $g\in L^2(\Omega,\Fcal_T,\PP)$,
\begin{equation*}
g=\mathbb{E}[g\mid\Fcal_0]+\int_{(0,T]}H_s^\top\,\mathrm{d}Y_s+L_T,
\end{equation*}
where $H$ is predictable, $H\cdot Y$ is a square-integrable martingale, and $L$ is a square-integrable martingale with $L_0=0$ that is strongly orthogonal to each coordinate of $Y$ \citep[p.~31]{ansel93decomposition}.
In the preceding example, $H=0$ and $L_t=(Z^2-2/3)\indicatorset{[t_0,\infty)}(t)$, since $\mathbb{E}[(Z^2-2/3)Z]=0$.

In contrast, Theorem~\ref{thm:path-approximation} approximates terminal claims by nonlinear functions of linear path observations, rather than representing them exactly as stochastic gains.
Moreover, part~\textup{(a)} requires neither a martingale representation property nor a semimartingale assumption on $Y$, and controls the error in the Orlicz norm on $M^\Phi$, or on $L^\Phi$ when $\Phi\in\Delta_2$.

Finally, taking $h_i^j=0$ and finitely supported measures $\mu_i^j=\sum_{k=1}^n a_{ik}^j\delta_{t_k}$, with $a_{ik}^j\in\re$ and $0\le t_1<\cdots<t_n\le T$, reduces \eqref{eq:path-net} to
\begin{equation*}
G=\sum_{i=1}^m\gamma_i\psi\biggl(
\sum_{k=1}^n\sum_{j=1}^d a_{ik}^jY_{t_k}^j+b_i\biggr),
\end{equation*}
the standard feedforward neural network with one hidden layer applied to $(Y_{t_1},\ldots,Y_{t_n})$.
The remainder of this section studies these networks and their approximation properties.
\end{remark}

The next proposition shows that even when there is an uncountable family of random variables generating the observable information, it suffices to reduce to a countable subfamily.
For an application of the following proposition, recall that in separable metric spaces, the Borel $\sigma$-algebra is always countably generated. 
\begin{proposition}\label{prop:countable-representation}
Let $(\Omega, \Fcal)$ and $(S, \mathcal{S})$ be measurable spaces.
Let $I$ be a non-empty set, and $(\Fcal_i)_{i \in I}$ be an indexed family of sub-$\sigma$-algebras of $\Fcal$.
For each $J \subset I$, denote $\Fcal_J = \sigma(\bigcup_{j \in J} \Fcal_j)$.
Let $g \from \Omega \to S$ be an $\Fcal_I$-$\mathcal{S}$-measurable function such that the trace-$\sigma$-algebra of $\mathcal{S}$ on $g(\Omega)$, which we denote $\mathcal{S}_g$, is countably generated.
Then there exists a countable set $J \subset I$ such that $g$ is $\Fcal_J$-$\mathcal{S}$-measurable.
\end{proposition}

\begin{proof}
Let $(S_n)_{n \in \na}$ be a sequence of elements of $\mathcal{S}$ such that $\mathcal{S}_g = \sigma(S_n \cap g(\Omega) \colon n \in \na)$.
For every $n \in \na$, let $F_n = g^{-1}(S_n) \in \Fcal_I$, and $\Fcal_g \coloneqq \sigma(F_n \colon n \in \na)$.
Then $g$ is $\Fcal_g$-$\mathcal{S}$-measurable, $\Fcal_g$ is countably generated, and $\Fcal_g \subset \Fcal_I$.
Next, we claim that
\begin{equation}\label{eq:sigma-algebra-countable-subfamilies}
\Fcal_I = \bigcup_{\substack{J \subset I \\ J\, \textrm{countable}}} \Fcal_J.
\end{equation}
To this end, first observe that the right-hand side of \eqref{eq:sigma-algebra-countable-subfamilies}, which we denote $\mathcal{G}$, is indeed a $\sigma$-algebra. 
For every countable $J \subset I$, clearly $\Omega \in \Fcal_J$, and thus $\Omega \in \mathcal{G}$. 
If $B \in \mathcal{G}$, then there exists a countable $J \subset I$ such that $B \in \Fcal_J$ and thus $B^{\mathsf{c}} \in \Fcal_J$, which implies that $B^{\mathsf{c}} \in \mathcal{G}$. 
Let $(B_n)_{n \in \na}$ be a sequence in $\mathcal{G}$.
For each $n \in \na$, there exists a countable $J_n \subset I$ such that $B_n \in \Fcal_{J_n}$.
The set $J = \bigcup_{n \in \na} J_n$ is countable, and $B_n \in \Fcal_J$ for each $n \in \na$, hence $\bigcup_{n \in \na} B_n \in \Fcal_J$ and thus $\bigcup_{n \in \na} B_n \in \mathcal{G}$.

This shows that $\mathcal{G}$ is indeed a $\sigma$-algebra. 
To show equality in \eqref{eq:sigma-algebra-countable-subfamilies}, note that $\bigcup_{i \in I} \Fcal_i$ generates $\Fcal_I$. 
Let $F \in \bigcup_{i \in I} \Fcal_i$. 
Then there exists $i \in I$ such that $F \in \Fcal_{\{i\}} \subset \mathcal{G}$, which implies $\Fcal_I \subset \mathcal{G}$. 
On the other hand, for each countable $J \subset I$, obviously $\Fcal_J \subset \Fcal_I$ and thus $\mathcal{G} \subset \Fcal_I$. 

To conclude, recall that $g$ is measurable with respect to the countably generated $\sigma$-algebra $\Fcal_g = \sigma(F_n \colon n \in \na)$, and that $\Fcal_g \subset \Fcal_I$. 
For each $n \in \na$, \eqref{eq:sigma-algebra-countable-subfamilies} implies the existence of a countable $J_n \subset I$ such that $F_n \in \Fcal_{J_n}$. 
Setting $J \coloneqq \bigcup_{n \in \na} J_n$ yields the assertion.
\end{proof}

The proof of Proposition~\ref{prop:countable-representation} used the fact that $g$ is $\Fcal_g$-measurable, where $\Fcal_g$ is countably generated.
If $S$ is a separable metric space, one may argue using simple functions.
Recall that a function $f \from \Omega \to S$ is simple if $f(\Omega) \subset S$ is a finite set.

\begin{lemma}
Let $(\Omega, \Fcal)$ be a measurable space, $(S, \varrho)$ be a separable metric space and $\B_S$ its Borel $\sigma$-algebra.
Let $g \from \Omega \to S$ be measurable.
Then there exists a sequence $(g_n)_{n \in \na}$ of $\Fcal$-$\B_S$-measurable simple functions such that $\lim_{n \to \infty} g_n(\omega) = g(\omega)$ for every $\omega \in \Omega$.
Moreover, there exists a countably generated sub-$\sigma$-algebra $\Fcal_g \subset \Fcal$ such that $g$ is $\Fcal_g$-$\B_S$-measurable.
\end{lemma}

\begin{proof}
The assertion is immediate if $\Omega=\varnothing$.
Otherwise, choose a dense sequence $(s_n)_{n \in \na}$ in $S$, with repetitions if necessary, and set $S_n = \{s_1, \ldots, s_n\}$ for $n \in \na$.
For every $n \in \na$, let
\begin{equation*}
g_n(\omega) = \argmin_{s \in S_n} \varrho(s, g(\omega)),\qquad \omega \in \Omega,
\end{equation*}
where ties are broken by the least index.
For $1 \le k \le n$, the set on which index $k$ is chosen is
\begin{equation*}
F_k^n \coloneqq\bigcap_{1 \le j < k}\{\varrho(s_k,g)<\varrho(s_j,g)\}\cap\bigcap_{k < j \le n}\{\varrho(s_k,g)\le\varrho(s_j,g)\}.
\end{equation*}
Since each $\varrho(s_j,g)$ is measurable, $F_k^n \in \Fcal$.
The sets $F_1^n, \ldots, F_n^n$ are pairwise disjoint and cover $\Omega$, with $g_n=s_k$ on $F_k^n$, so $g_n$ is simple and $\Fcal$-$\B_S$-measurable.
For each $\omega \in \Omega$ and $\varepsilon > 0$, there exists $m \in \na$ such that $\varrho(s_m, g(\omega)) < \varepsilon$.
Then $\varrho(g_n(\omega), g(\omega)) \le \varrho(s_m,g(\omega)) < \varepsilon$ for every $n \ge m$, which proves the pointwise convergence.
The $\sigma$-algebra
\begin{equation*}
\Fcal_g \coloneqq\sigma(F_k^n \colon n \in \na,\ 1 \le k \le n)
\end{equation*}
is countably generated and contained in $\Fcal$.
Each $g_n$ is $\Fcal_g$-$\B_S$-measurable.
Since measurability carries over to the pointwise limit \citep[Lemma~1.11(ii)]{kallenberg21foundations}, $g$ is $\Fcal_g$-$\B_S$-measurable.
\end{proof}

In Theorem~\ref{thm:path-approximation}, we considered the approximation of random variables that are measurable with respect to the $\sigma$-algebra $\Fcal_T$ at a fixed time $T>0$.
More generally we consider random variables that are measurable with respect to the stopping-time $\sigma$-algebra $\Fcal_\tau$ for $\FF^0$- or $\FF$-stopping times $\tau$.
The latter case is particularly relevant when completion renders $\FF$ right-continuous, so that it satisfies the usual hypotheses.
The following assumption covers these cases, where conditions~\textup{(a)} and~\textup{(b)} are invoked separately for $\FF^0$- or $\FF$-stopping times, respectively.
Henceforth, we write $Y_s^\tau(\omega)\coloneqq Y_{s\wedge\tau(\omega)}(\omega)$ for $s\ge0$ and $\omega\in\Omega$.
\begin{assumption}\label{ass:stopped-path-representation}
At least one of the following conditions holds:
\begin{enumerate}
\item[{(a)}] The map $(\omega,t)\mapsto Y_t(\omega)$ is $(\Fcal_\infty^0\otimes\B_{\replus})$-$\B_{\re^d}$-measurable.
For all $t\ge0$, $\omega\in\Omega$, there exists $\widetilde\omega\in\Omega$ such that $Y_s(\widetilde\omega)=Y_{s\wedge t}(\omega)$ for every $s\ge0$.
\item[{(b)}] The process $Y$ is a L\'{e}vy process.
\end{enumerate}
\end{assumption}
The second part of condition~\textup{(a)} is called the sufficient-richness condition \citep[Section~1.2]{shiryaev08stopping}.
It holds for the coordinate process on the spaces of continuous or c\`{a}dl\`{a}g paths, and covers canonical realizations of continuous and jump price models.
Condition~\textup{(b)} is motivated by \citet[Theorem~13.45]{he92semimartingale} and includes Brownian motion and L\'{e}vy log-price processes.
In particular, exponential L\'{e}vy price models are covered by taking $Y$ to be the underlying L\'{e}vy log-price process.
The two conditions are complementary: condition~\textup{(a)} concerns the realization of the process on the underlying sample space, whereas condition~\textup{(b)} concerns its stochastic dynamics.
Neither implies the other in general, although both hold for canonical path-space realizations of many L\'{e}vy models.

Before we state and prove our next result, we use the convention that for a stopping time $\tau$ of either $\mathbb F^0$ or $\mathbb F$, we write $\mathcal F_\tau$ for the stopping-time $\sigma$-algebra associated with the filtration with respect to which $\tau$ is being considered.
\begin{proposition}\label{prop:countable-observations}
In the context of Proposition~\ref{prop:countable-representation}, and for our complete probability space $(\Omega, \Fcal, \PP)$ with filtrations $\FF^0$ and $\FF$,
\begin{enumerate}
\item[{(a)}] Fix $t \in [0, \infty]$, and let $g \from \Omega \to \re$ be $\Fcal_t^0$-measurable.
Then there exists a countable set $J \subset [0, t] \cap \mathbb{R}_+$ such that, denoting $\Fcal_J = \sigma(Y_t \colon t \in J)$, $g$ is $\Fcal_J$-measurable.
\item[{(b)}] Suppose that Assumption~\ref{ass:stopped-path-representation}(a) holds.
Let $\tau\from\Omega\to\replus$ be an $\FF^0$-stopping time, and let $g\from\Omega\to\re$ be $\Fcal_\tau$-measurable.
Then there exists a countable set $J\subseteq\replus$ such that $g$ is measurable with respect to $\Fcal_J=\sigma(Y_s^\tau\colon s\in J)$.
\item[{(c)}] Suppose that Assumption~\ref{ass:stopped-path-representation}(b) holds.
Let $\tau\from\Omega\to\replus$ be an $\FF$-stopping time, and let $g\from\Omega\to\re$ be $\Fcal_\tau$-measurable.
Then there exists a countable set $J\subseteq\replus$ such that $g$ is measurable with respect to $\Fcal_J=\sigma(\tau)\vee\sigma(Y_s^\tau\colon s\in J)\vee\sigma(\mathcal N_\PP)$.
\end{enumerate}
\end{proposition}

\begin{proof}
Part~\textup{(a)} follows from Proposition~\ref{prop:countable-representation} with $I=[0,t]\cap\replus$ and $\Fcal_i=\sigma(Y_i)$.
For part~\textup{(b)}, under Assumption~\ref{ass:stopped-path-representation}(a),
\begin{equation*}
\Fcal_\tau=\sigma(Y_s^\tau\colon s\ge0)
\end{equation*}
by \citet[Section~1.2, Theorem~6]{shiryaev08stopping}.
Proposition~\ref{prop:countable-representation}, applied with $I=\replus$ and $\Fcal_i=\sigma(Y_i^\tau)$, therefore yields the assertion.
For part~\textup{(c)}, \citet[Theorem~13.45]{he92semimartingale} gives
\begin{equation*}
\Fcal_\tau=\sigma(\tau)\vee\sigma(Y_s^\tau\colon s\ge0)\vee\sigma(\mathcal N_\PP).
\end{equation*}
For $s\ge0$, let $\mathcal G_s\coloneqq\sigma(\tau)\vee\sigma(Y_s^\tau)\vee\sigma(\mathcal N_\PP)$.
Then $\Fcal_\tau=\sigma(\mathcal G_s\colon s\ge0)$.
Proposition~\ref{prop:countable-representation} therefore yields a countable set $J\subseteq\replus$ such that $g$ is measurable with respect to
\begin{equation*}
\sigma(\mathcal G_s\colon s\in J)\subseteq\sigma(\tau)\vee\sigma(Y_s^\tau\colon s\in J)\vee\sigma(\mathcal N_\PP).
\end{equation*}
\end{proof}

\begin{definition}\label{def:neural-random-variables}
Let $\tau\from\Omega\to\replus$ be a stopping time for either $\FF^0$ or $\FF$.
An $\Fcal_\tau$-measurable function $\varphi\from\Omega\to\re$ is represented by a neural network if there exist $n\in\na$, deterministic times $0\le t_1<\cdots<t_n<\infty$, and a neural network
$f\in\mathcal{NN}((\re^{1+d\times n})^\ast,\psi)$ such that
\begin{equation*}
\varphi=f(\tau,Y_{t_1}^\tau,\ldots,Y_{t_n}^\tau).
\end{equation*}
We denote the set of these functions by $\mathcal{NN}_\tau(\psi)$.
For a deterministic time $t<\infty$, we write $\mathcal{NN}_t(\psi)\coloneqq\mathcal{NN}_\tau(\psi)$ with $\tau\equiv t$.
\end{definition}
For deterministic times $t$, the constant time coordinate can be absorbed into the bias terms, while under Assumption~\ref{ass:stopped-path-representation}(a) the time coordinate may be omitted by setting the corresponding weights to zero.

In Theorem~\ref{thm:path-approximation} we showed that neural networks that act on the entire sample path of $Y$ are dense in the corresponding Orlicz space.
Proposition~\ref{prop:countable-observations} then showed that random variables measurable with respect to the natural or stopping-time $\sigma$-algebras are determined by at most countably many observations of the stopped process $Y$.
For such random variables, we instead restrict attention to finitely many observations.
The following theorem shows that this suffices to render random variables represented by neural networks dense in the corresponding Orlicz spaces.
It extends \citet[Lemma~4.3.1]{oksendal03equations}, where $L^2$-random variables measurable with respect to the Brownian filtration at a fixed time are approximated by functions of finitely many Brownian observations.
We continue to work under Standing Assumption~\ref{itm:sigmoidal-activation}.
In Section~\ref{sec:stochastic-integration}, case~\textup{(b)(iii)} will be of particular importance. 

\begin{theorem}\label{thm:random-variable-approximation}
Let $(\Phi,\Psi)$ be a complementary pair of Young functions, where $\Phi\in\Delta_2$.
\begin{enumerate}
\item[{(a)}] Let $m\in\na$, let $I$ be a non-empty set, and let $(Z_i)_{i\in I}$ be an indexed family of $\Fcal$-measurable random variables $Z_i\from\Omega\to\re^m$.
Let $\widetilde{\Fcal}=\sigma(Z_i\colon i\in I)\vee\sigma(\mathcal N_\PP)$, and let $g\in\mathcal L^\Phi(\PP)$ be $\widetilde{\Fcal}$-measurable.
Then for every $\varepsilon>0$, there exist $n\in\na$, a finite set $J=\{j_1,\ldots,j_n\}\subset I$, and a neural network $f\in\mathcal{NN}((\re^{m\times n})^\ast,\psi)$ such that
\begin{equation*}
\|g-f(Z_{j_1},\ldots,Z_{j_n})\|_\Phi<\varepsilon.
\end{equation*}
\item[{(b)}] Let $\tau\from\Omega\to\replus$ and suppose that one of the following conditions holds:
\begin{enumerate}
\item[{(i)}] $\tau$ is a deterministic time;
\item[{(ii)}] Assumption~\ref{ass:stopped-path-representation}(a) holds and $\tau$ is an $\FF^0$-stopping time;
\item[{(iii)}] Assumption~\ref{ass:stopped-path-representation}(b) holds and $\tau$ is an $\FF$-stopping time.
\end{enumerate}
Let $g\in\mathcal L^\Phi(\PP)$ be $\Fcal_\tau$-measurable.
Then for every $\varepsilon>0$, there exists $f\in\mathcal{NN}_\tau(\psi)$ such that
\begin{equation*}
\|g-f\|_\Phi<\varepsilon.
\end{equation*}
\end{enumerate}
\end{theorem}

\begin{proof}
For part~\textup{(a)}, let $\mathcal G\coloneqq\sigma(Z_i\colon i\in I)$.
By Lemma~\ref{lem:orlicz-l1}, $g\in L^1(\PP)$.
Since $\widetilde{\Fcal}=\mathcal G\vee\sigma(\mathcal N_\PP)$, adjoining the $\PP$-null sets does not change random variables up to $\PP$-almost sure equality.
Hence $\widetilde g\coloneqq\mathbb E[g\mid\mathcal G]$ is $\mathcal G$-measurable and satisfies $\widetilde g=g$ $\PP$-almost surely.
By Proposition~\ref{prop:countable-representation}, there exists a countable set $J\subseteq I$ such that $\widetilde g$ is measurable with respect to $\Fcal_J\coloneqq\sigma(Z_j\colon j\in J)$.
Since $I$ is non-empty, we may enlarge $J$ if necessary so that $J\neq\varnothing$, and then write $J=\{j_k\colon k\in\na\}$, allowing repetitions if $J$ is finite.
For $n\in\na$, let $\mathcal G_n\coloneqq\sigma(Z_{j_1},\ldots,Z_{j_n})$.
Then $(\mathcal G_n)_{n\in\na}$ is increasing and $\sigma(\bigcup_{n\in\na}\mathcal G_n)=\Fcal_J$.
Set $M_n\coloneqq\mathbb E[\widetilde g\mid\mathcal G_n]$.
By Theorem~\ref{thm:orlicz-conditional-expectations},
\begin{equation*}
M_n\rightarrow\widetilde g\qquad\text{in }L^\Phi(\PP).
\end{equation*}
Choose $n\in\na$ such that $\|\widetilde g-M_n\|_\Phi<\varepsilon/2$.
By the Doob--Dynkin factorization lemma \citep[Lemma~1.14]{kallenberg21foundations}, there exists a Borel-measurable function $h\colon\re^{mn}\to\re$ such that $M_n=h(Z_{j_1},\ldots,Z_{j_n})$.
Let $\mu$ be the distribution of $(Z_{j_1},\ldots,Z_{j_n})$ on $\re^{mn}$.
Since $M_n\in L^\Phi(\PP)$, we have $h\in L^\Phi(\mu)$.
By Theorem~\ref{thm:orlicz-approximation}, there exists $f\in\mathcal{NN}((\re^{m\times n})^\ast,\psi)$ such that $\|h-f\|_\Phi<\varepsilon/2$ with respect to $\mu$.
By the definition of $\mu$, $\|M_n-f(Z_{j_1},\ldots,Z_{j_n})\|_\Phi=\|h-f\|_\Phi$.
Since $g=\widetilde g$ $\PP$-almost surely, the triangle inequality gives
\begin{equation*}
\|g-f(Z_{j_1},\ldots,Z_{j_n})\|_\Phi<\varepsilon.
\end{equation*}

For part~\textup{(b)(i)}, let $\tau\equiv t$.
Since $\Fcal_t=\sigma(Y_s\colon 0\le s\le t)\vee\sigma(\mathcal N_\PP)$, part~\textup{(a)} applies with $m=d$, $I=[0,t]$, and $Z_s=Y_s$.
Hence there exist $0\le t_1<\cdots<t_n\le t$ and $f\in\mathcal{NN}((\re^{d\times n})^\ast,\psi)$ such that $\|g-f(Y_{t_1},\ldots,Y_{t_n})\|_\Phi<\varepsilon$.
Define $\widehat f(r,x_1,\ldots,x_n)\coloneqq f(x_1,\ldots,x_n)$.
Since $\widehat f$ is obtained from $f$ by precomposition with a linear projection, $\widehat f\in\mathcal{NN}((\re^{1+d\times n})^\ast,\psi)$.
As $\tau=t$ and $Y_{t_i}^\tau=Y_{t_i}$, the random variable $\widehat f(\tau,Y_{t_1}^\tau,\ldots,Y_{t_n}^\tau)$ belongs to $\mathcal{NN}_\tau(\psi)$.

For part~\textup{(b)(ii)}, Proposition~\ref{prop:countable-observations}(b) yields a countable set $J\subseteq\replus$ such that $g$ is measurable with respect to $\sigma(Y_s^\tau\colon s\in J)$.
Replacing $J$ by $J\cup\{0\}$ if necessary, we may assume that $J\neq\varnothing$.
Applying part~\textup{(a)} with $m=d$ to the family $(Y_s^\tau)_{s\in J}$ yields deterministic times $0\le t_1<\cdots<t_n<\infty$ and $f\in\mathcal{NN}((\re^{d\times n})^\ast,\psi)$ such that $\|g-f(Y_{t_1}^\tau,\ldots,Y_{t_n}^\tau)\|_\Phi<\varepsilon$.
Define again $\widehat f(r,x_1,\ldots,x_n)\coloneqq f(x_1,\ldots,x_n)$.
Then as in the previous case, $\widehat f\in\mathcal{NN}((\re^{1+d\times n})^\ast,\psi)$ and $\widehat f(\tau,Y_{t_1}^\tau,\ldots,Y_{t_n}^\tau)\in\mathcal{NN}_\tau(\psi)$.

For part~\textup{(b)(iii)}, Proposition~\ref{prop:countable-observations}(c) gives a countable set $J\subseteq\replus$ such that $g$ is measurable with respect to $\sigma(\tau)\vee\sigma(Y_s^\tau\colon s\in J)\vee\sigma(\mathcal N_\PP)$.
Replacing $J$ by $J\cup\{0\}$ if necessary, we may assume that $J\neq\varnothing$.
For $s\in J$, set $Z_s\coloneqq(\tau,Y_s^\tau)\in\re^{d+1}$.
Applying part~\textup{(a)} with $m=d+1$ to the family $(Z_s)_{s\in J}$ gives deterministic times $0\le t_1<\cdots<t_n<\infty$ and $f\in\mathcal{NN}((\re^{(d+1)\times n})^\ast,\psi)$ such that
\begin{equation*}
\|g-f((\tau,Y_{t_1}^\tau),\ldots,(\tau,Y_{t_n}^\tau))\|_\Phi<\varepsilon.
\end{equation*}
Define the linear map $L\colon\re^{1+d\times n}\to\re^{(d+1)\times n}$ by $L(r,x_1,\ldots,x_n)=((r,x_1),\ldots,(r,x_n))$.
Since precomposition with $L$ maps linear functionals to linear functionals, $f\circ L\in\mathcal{NN}((\re^{1+d\times n})^\ast,\psi)$.
Therefore $(f\circ L)(\tau,Y_{t_1}^\tau,\ldots,Y_{t_n}^\tau)\in\mathcal{NN}_\tau(\psi)$.
\end{proof}

Recall that throughout, $(\Omega, \Fcal, \PP)$ is a complete probability space, and the initial $\sigma$-algebra $\Fcal_0$ is assumed to contain all $\PP$-null sets of $\Fcal$. 
In continuous-time mathematical finance, we endow this space with a filtration $\FF$ such that $(\Omega, \Fcal, \FF, \PP)$ satisfies the usual hypotheses, i.e., $\FF$ is complete and right-continuous.
While completeness is inherited automatically from our standing assumption, right-continuity is the remaining issue.
A sufficient condition follows from \citet[Remark~3.23]{schmock24stochastic}: if a process $Y$ is right-continuous with respect to the Euclidean norm on $\re^d$ and has independent future increments relative to its natural filtration $\FF^0$, then the associated filtration $\FF$ is right-continuous.
Brownian motion and, more generally, every L\'{e}vy process satisfy both requirements and are therefore admissible choices for $Y$.
This complements \citet[Theorem~13.45]{he92semimartingale}, used in Proposition~\ref{prop:countable-observations}(c), which identifies the corresponding stopping-time $\sigma$-algebras, and motivates the following standing assumption on the driving process $Y$.
\begin{assumption}\label{ass:levy-driver} 
The process $Y$ is a L\'{e}vy process.
Consequently, the filtered probability space $(\Omega, \Fcal, \FF, \PP)$ is complete and satisfies the usual hypotheses. 
\end{assumption}

\section{Universal approximation and stochastic integration}\label{sec:stochastic-integration}
Theorem~\ref{thm:random-variable-approximation} showed that measurable trading positions can be approximated by neural networks using observations of $Y$.
In a continuous-time market, however, a simple trading strategy consists of a sequence of such positions: at each trading time, the available information determines a position that is held until the next decision is made.
The positions of a strategy $V$ generate trading gains.
For a price process $X$, the cumulative gain is then given by the stochastic integral $\int V\,\mathrm{d}X$.
We next show that the approximation of individual trading decisions extends to arbitrary predictable stochastic integrals and, for c\`{a}gl\`{a}d integrands, the integrands themselves.

\begin{definition}
A real-valued stochastic process $V$ is called a simple predictable process if it admits a representation
\begin{equation}\label{eq:simple-process}
V_t = \varphi_0\indicatorset{\{0\}}(t) + \sum_{i=1}^n\varphi_i\indicatorset{(\tau_i,\tau_{i+1}]}(t), \qquad t\ge0,
\end{equation}
where $n\in\na$, $0=\tau_0=\tau_1\le\cdots\le\tau_{n+1}<\infty$ are $\FF$-stopping times, and $\varphi_i\in\mathcal{L}^\infty(\Fcal_{\tau_i},\PP)$ for $i=0,\ldots,n$.
A simple predictable process is called algorithmic if it admits such a representation with $\varphi_i\in\mathcal{NN}_{\tau_i}(\psi)$ for $i=0,\ldots,n$.
We denote the classes of simple and simple algorithmic strategies by $\mathcal{S}$ and $\mathcal{S}(\psi)$, respectively.
Requiring deterministic times in the respective representations yields the subclasses $\mathcal{E}$ and $\mathcal{E}(\psi)$ of elementary and elementary algorithmic strategies.
\end{definition}

\begin{proposition}\label{prop:simple-process-approximation}
Let $(\Phi, \Psi)$ be a complementary pair of Young functions, where $\Phi \in \Delta_2$.
Then for each $V \in \mathcal{E}$ and $\varepsilon > 0$, there exists $U \in \mathcal{E}(\psi)$ such that $\| (V-U)_\infty^\ast \|_\Phi < \varepsilon$.
In particular, $\mathcal{E}(\psi)$ is dense in $\mathcal{E}$ for the ucp topology.
Moreover, for each $V \in \mathcal{S}$ and $\varepsilon > 0$, there exists $U \in \mathcal{S}(\psi)$ such that $\| (V-U)_\infty^\ast \|_\Phi < \varepsilon$.
In particular, $\mathcal{S}(\psi)$ is dense in $\mathcal{S}$ for the ucp topology.
\end{proposition}

\begin{proof}
We argue both cases simultaneously.
Let $V$ be given by a representation~\eqref{eq:simple-process}, either with deterministic times $\tau_i$ or with $\FF$-stopping times $\tau_i$.
Fix $i \in \{0, 1, \ldots, n\}$, and note that $\varphi_i \in \mathcal{L}^\infty(\PP) \subset \mathcal{L}^\Phi(\PP)$.
By Theorem~\ref{thm:random-variable-approximation}(b)(i) and (b)(iii), respectively, there exists a sequence $(\varphi_{i,m})_{m\in\na}$ in $\mathcal{NN}_{\tau_i}(\psi)$ that converges to $\varphi_i$ in $\| \cdot \|_\Phi$.

Let $(V^m)_{m \in \na}$ be the sequence given by
\begin{equation*}
V^{m} = \varphi_{0,m}\indicatorset{\{0\}}(\cdot) + \sum_{i=1}^n \varphi_{i,m} \indicatorset{(\tau_i, \tau_{i+1}]}(\cdot), \quad m \in \na.
\end{equation*}
Since $\psi$ is bounded, each $\varphi_{i,m}$ is bounded, so $V^m\in\mathcal E(\psi)$ or $V^m\in\mathcal S(\psi)$, respectively.
By construction, $(V-V^m)_\infty^\ast = \sup_{t\ge0}|V_t-V_t^m| \le \sum_{i=0}^n |\varphi_i - \varphi_{i,m}|$.
By \citet[Proposition~3.3.4]{rao91orlicz}, the Orlicz norm $\| \cdot \|_\Phi$ is monotone in the sense that $0 \le f_1 \le f_2$ implies $\|f_1\|_\Phi \le \|f_2\|_\Phi$, hence 
\begin{equation*}
\| (V-V^m)_\infty^\ast \|_\Phi \le \big\| \sum_{i=0}^n |\varphi_i - \varphi_{i,m}| \big\|_\Phi \le \sum_{i=0}^n \| \varphi_i - \varphi_{i,m} \|_\Phi,
\end{equation*}
where the right-hand side converges to zero as $m \to \infty$.
Therefore, there exists $m \in \na$ such that $\| (V-V^m)_\infty^\ast \|_\Phi < \varepsilon$, and the required process $U$ is given by $U = V^m$.

By Lemma~\ref{lem:orlicz-mean-convergence}, norm convergence implies convergence in probability, so, for every $\varepsilon>0$,
\begin{equation*}
\PP\big( (V-V^m)_\infty^\ast > \varepsilon \big) \rightarrow 0, \qquad m\to\infty,
\end{equation*}
hence $V^m$ converges to $V$ uniformly on $\replus$ in probability, in particular in the ucp topology.
\end{proof}

\begin{example}\label{ex:algorithmic-class-small}
Proposition~\ref{prop:simple-process-approximation} shows that $\mathcal E(\psi)$ is dense in $\mathcal E$.
Nevertheless, the restriction from $\mathcal E$ to $\mathcal E(\psi)$ can be substantial.
To illustrate this point by means of an example, let $Y=W$ be a one-dimensional Brownian motion with its augmented natural filtration, and fix $0<t<T$.
For $h\in L^2([0,t])$, define
\begin{equation}\label{eq:brownian-elementary-family}
Z_h \coloneqq \int_0^t h(s)\,\mathrm{d}W_s,\qquad H^h \coloneqq \tanh(Z_h)\indicatorset{(t,T]}.
\end{equation}
Since $\tanh(Z_h)$ is bounded and $\Fcal_t$-measurable, $H^h\in\mathcal E$ for every $h\in L^2([0,t])$.
By contrast, we show that
\begin{equation}\label{eq:algorithmic-kernels}
\bigl\{h\in L^2([0,t]) : H^h\in\mathcal E(\psi)\bigr\}
\end{equation}
is contained in a countable union of compact subsets of $L^2([0,t])$, each having empty interior.

Suppose that $H^h\in\mathcal E(\psi)$.
The case $h=0$ is immediate, so let $h\ne0$ and fix an elementary algorithmic representation of $H^h$.
Choose $s\in(t,T)$ which is not one of its finitely many trading times.
The coefficient on the interval containing $s$ then equals $\tanh(Z_h)$ almost surely.
By the definition of $\mathcal E(\psi)$, there exist $0\le s_1,\ldots,s_n\le T$ such that $\tanh(Z_h)$ is measurable with respect to $\sigma(W_{s_1},\ldots,W_{s_n})$.
Since $\tanh$ is one-to-one, the same holds for $Z_h$.

Let $\mathcal H\coloneqq\operatorname{span}\{W_{s_1},\ldots,W_{s_n}\}\subset L^2(\PP)$ and denote by $PZ_h$ the orthogonal projection of $Z_h$ onto $\mathcal H$.
Then $R\coloneqq Z_h-PZ_h$ is Gaussian and uncorrelated with every $W_{s_i}$.
Indeed, since $PZ_h\in\mathcal H$, we may write $PZ_h=\sum_{i=1}^n a_iW_{s_i}$, and hence $R=\int_0^T\left(h(u)\mathbf 1_{[0,t]}(u)-\sum_{i=1}^n a_i\mathbf 1_{[0,s_i]}(u)\right)\,\mathrm dW_u$, so $(R,W_{s_1},\ldots,W_{s_n})$ is jointly Gaussian.
Therefore, $R$ is independent of $\sigma(W_{s_1},\ldots,W_{s_n})$.
It is also measurable with respect to this $\sigma$-algebra, because both $Z_h$ and $PZ_h$ are measurable with respect to it.
Hence $R$ is almost surely constant, and $\mathbb E[R]=0$ gives $R=0$ almost surely.
Thus
\begin{equation}\label{eq:gaussian-span}
Z_h\in\operatorname{span}\{W_{s_1},\ldots,W_{s_n}\}.
\end{equation}
Writing $Z_h=\int_0^T h(u)\indicatorset{[0,t]}(u)\,\mathrm{d}W_u$ and $W_{s_i}=\int_0^T\indicatorset{[0,s_i]}(u)\,\mathrm{d}W_u$, the It\^{o} isometry implies that $h\indicatorset{[0,t]}$ belongs to the linear span of $\indicatorset{[0,s_1]},\ldots,\indicatorset{[0,s_n]}$ in $L^2([0,T])$, since for suitable $a_1,\ldots,a_n\in\mathbb R$, $\|h\indicatorset{[0,t]}-\sum_{i=1}^n a_i\indicatorset{[0,s_i]}\|_2^2 =\mathbb E[(Z_h-\sum_{i=1}^n a_iW_{s_i})^2]=0$.
Consequently, $h$ agrees almost everywhere on $[0,t]$ with a finite step function.

For $n,m\in\na$, let $K_{n,m}$ consist of all functions of the form
\begin{equation}\label{eq:bounded-step-functions}
\sum_{j=1}^n a_j\indicatorset{(r_{j-1},r_j]},\qquad 0=r_0\le\cdots\le r_n=t,\qquad |a_j|\le m.
\end{equation}
Every finite step function belongs to some $K_{n,m}$.
Moreover, $K_{n,m}$ is compact in $L^2([0,t])$: the admissible coefficients and partition points form a compact subset of Euclidean space, and the map from these parameters to~\eqref{eq:bounded-step-functions} is continuous in $L^2$.
Indeed, convergence of the coefficients and partition points implies $L^2$-convergence because the coefficients remain uniformly bounded and the lengths of the symmetric differences of the corresponding intervals tend to zero.

Each $K_{n,m}$ has empty interior.
Otherwise, its compactness would imply that $\overline B(f,r)$ is compact for some $f\in L^2([0,t])$ and $r>0$, and hence so is $\overline B(0,1)=r^{-1}\bigl(\overline B(f,r)-f\bigr)$.
This is impossible: for pairwise disjoint intervals $A_k\subset[0,t]$ of positive length, the functions $e_k\coloneqq\indicatorset{A_k}/\sqrt{|A_k|}$ belong to the closed unit ball and satisfy $\|e_k-e_\ell\|_{L^2}=\sqrt{2}$ for $k\ne\ell$, so they have no convergent subsequence.
The set in~\eqref{eq:algorithmic-kernels} is therefore contained in $\bigcup_{n,m\in\na}K_{n,m}$.
Since $L^2([0,t])$ is complete, the Baire category theorem implies that this countable union cannot contain any non-empty open subset.

There is also a measure-theoretic counterpart to this conclusion.
Recall that a subset $A$ of a Banach space $E$ is called shy if it is contained in a Borel set $B$ for which there exists a compactly supported Borel probability measure $\mu$ satisfying $\mu(x+B)=0$ for every $x\in E$; the complement of a shy set is called prevalent.
In finite-dimensional spaces, shyness is equivalent to having Lebesgue measure zero.
Moreover, every compact subset of an infinite-dimensional Banach space is shy, and every countable union of shy sets is shy; see \citet{hunt92prevalence}.
Since $L^2([0,t])$ is infinite-dimensional, each $K_{n,m}$ is shy, and hence so is the set in~\eqref{eq:algorithmic-kernels}.
Consequently,
\begin{equation}\label{eq:prevalent-nonalgorithmic}
\bigl\{h\in L^2([0,t]) : H^h\in\mathcal E\setminus\mathcal E(\psi)\bigr\}
\end{equation}
is prevalent in $L^2([0,t])$.

Hence $\mathcal E$ contains the entire family~\eqref{eq:brownian-elementary-family}, indexed by the infinite-dimensional space $L^2([0,t])$, whereas membership in $\mathcal E(\psi)$ is possible only for a set of parameters that is both contained in a countable union of compact sets with empty interior and shy.
Thus the passage from $\mathcal E$ to $\mathcal E(\psi)$ constitutes a substantial restriction, even though Proposition~\ref{prop:simple-process-approximation} shows that $\mathcal E(\psi)$ remains dense in $\mathcal E$ and can approximate every $H^h$ in the ucp topology.
\end{example}

Elements of $\mathcal{S}(\psi)$ -- or, more generally, of $\mathcal{S}$ -- admit an elementary stochastic integral.
Denote by $\mathcal{L}$ and $\mathcal{D}$ the spaces of real-valued càglàd and càdlàg $\FF$-adapted processes, respectively, identified up to indistinguishability and endowed with the topology $\T_{\mathrm{ucp}}$ of ucp-convergence.
For $X\in\mathcal{D}$ and $V\in\mathcal{S}$ with representation~\eqref{eq:simple-process}, define the elementary integral $I(V,X)=\int V\,\mathrm{d}X$ by
\begin{equation*}
I(V,X)_t \coloneqq \varphi_0X_0 + \sum_{i=1}^n \varphi_i \bigl(X_{t\wedge\tau_{i+1}}-X_{t\wedge\tau_i}\bigr), \qquad t\ge0.
\end{equation*}
The definition does not depend on the representation of $V$, and it yields an adapted càdlàg process, with $I(V,X)_0=V_0X_0$.
For a semimartingale $X$, let $L(X)$ denote the space of real-valued predictable $X$-integrable processes \citep[p.~130]{chou80integrales}.
We write $V\cdot X$ for the stochastic integral and retain the convention
\begin{equation*}
(V\cdot X)_t=V_0X_0+\int_{(0,t]}V_s\,\mathrm dX_s,\qquad t\ge0,
\end{equation*}
so that $V\cdot X=I(V,X)$ for $V\in\mathcal S$.
Every $V\in\mathcal L$ belongs to $L(X)$ \citep[Chapter~II, Section~4]{protter05integration}.
Let $\mathcal{SM}\subseteq\mathcal D$ denote the space of semimartingales, equipped with the semimartingale topology $\T_{\mathrm{sm}}$ \citep[p.~264]{emery79topologie}, which is metrizable and finer than $\T_{\mathrm{ucp}}$.

In continuous-time financial models, trading strategies are naturally required to be predictable: a portfolio position may depend on the past information, but not on the subsequent price movement.
Thus, extending the preceding approximation theory to general predictable integrands requires showing that algorithmic strategies generate (up to indistinguishable modifications) the predictable $\sigma$-algebra.
The next proposition establishes this property.

\begin{proposition}\label{prop:predictable-sigma-algebra}
Let $\mathcal P$ denote the predictable $\sigma$-algebra on $\Omega\times\replus$, and set
\begin{equation*}
\mathcal E_0\coloneqq\{Z\in\mathcal E\colon Z\text{ is indistinguishable from zero}\}.
\end{equation*}
For a class $\mathcal C$ of processes, let $\sigma(\mathcal C)$ denote the $\sigma$-algebra generated by the maps $(\omega,t)\mapsto V_t(\omega)$, $V\in\mathcal C$.
Then
\begin{equation*}
\mathcal P=\sigma(\mathcal E)=\sigma\bigl(\mathcal E(\psi)+\mathcal E_0\bigr)=\sigma\bigl(\mathcal S(\psi)+\mathcal E_0\bigr),
\end{equation*}
where $\mathcal C+\mathcal E_0\coloneqq\{U+Z\colon U\in\mathcal C,\ Z\in\mathcal E_0\}$.
\end{proposition}

\begin{proof}
Recall that the predictable $\sigma$-algebra $\mathcal P$ is generated by the real-valued left-continuous adapted processes.
Since every elementary predictable process is left-continuous and adapted, we have $\sigma(\mathcal E)\subseteq\mathcal P$.
For the reverse inclusion, let $H$ be a real-valued left-continuous adapted process and put $c_n(x)\coloneqq(-n)\vee(x\wedge n)$ for $n\in\na$.
The processes
\begin{equation*}
H^n_t=c_n(H_0)\indicatorset{\{0\}}(t)+\sum_{k=0}^{n2^n-1}c_n(H_{k2^{-n}})\indicatorset{(k2^{-n},(k+1)2^{-n}]}(t),\qquad t\ge0,
\end{equation*}
belong to $\mathcal E$ and converge to $H$ at every point of $\Omega\times\replus$ by left-continuity, with convergence at time zero following from the initial term.
Thus $H$ is $\sigma(\mathcal E)$-measurable, and hence $\mathcal P=\sigma(\mathcal E)$.

Put $\mathcal G\coloneqq\sigma(\mathcal E(\psi)+\mathcal E_0)$ and fix $V\in\mathcal E$.
By Proposition~\ref{prop:simple-process-approximation} with $\Phi(x)=x^2$, choose $U^n\in\mathcal E(\psi)$ such that
\begin{equation*}
\mathbb E\bigl[((V-U^n)_\infty^\ast)^2\bigr]<2^{-n},\qquad n\in\na.
\end{equation*}
By Tonelli's theorem, the nonnegative random variable $\sum_{n=1}^{\infty}((V-U^n)_\infty^\ast)^2$ has finite expectation and is therefore finite almost surely.
Consequently, $(V-U^n)_\infty^\ast\to0$ almost surely.
Hence there is $N\in\mathcal N_\PP$ such that $U^n(\omega)$ converges uniformly on $\replus$ to $V(\omega)$ for every $\omega\notin N$.

Set $Z^n_t\coloneqq(V_t-U^n_t)\indicatorset{N}$ for $t\ge0$.
Since $N\in\Fcal_0$, a common refinement of the deterministic time partitions of $V$ and $U^n$ shows that $Z^n\in\mathcal E_0$.
The processes $U^n+Z^n$ agree with $V$ on $N\times\replus$ and converge to $V$ at every point of $N^{\mathsf c}\times\replus$.
They are $\mathcal G$-measurable, so their pointwise limit $V$ is $\mathcal G$-measurable.
Thus $\mathcal P=\sigma(\mathcal E)\subseteq\mathcal G$.

Finally, $\mathcal E(\psi)\subseteq\mathcal S(\psi)$ and every process in $\mathcal S(\psi)+\mathcal E_0$ is predictable.
Therefore
\begin{equation*}
\mathcal P=\mathcal G\subseteq\sigma\bigl(\mathcal S(\psi)+\mathcal E_0\bigr)\subseteq\mathcal P,
\end{equation*}
which proves the claim.
\end{proof}

Propositions~\ref{prop:simple-process-approximation} and~\ref{prop:predictable-sigma-algebra} now yield our main theorem, restated here for convenience.
\begin{theorem}\label{thm:ucp-approximation}
Let $X$ be a c\`{a}dl\`{a}g $\FF$-semimartingale and let $V\in L(X)$.
Then there exists a sequence $(V^n)_{n\in\na}\subset\mathcal E(\psi)$ such that
\begin{equation*}
\int V_s^n\, \mathrm{d}X_s \to \int V_s\, \mathrm{d}X_s, \quad n \to \infty,
\end{equation*}
with convergence in $\T_{\mathrm{sm}}$.
If $V\in\mathcal L$, then the sequence $(V^n)_{n\in\na}$ may be chosen in $\mathcal S(\psi)$, such that $V^n\to V$ in $\T_{\mathrm{ucp}}$ holds in addition to $V^n\cdot X\rightarrow V\cdot X$ in $\T_{\mathrm{sm}}$.
\end{theorem}

\begin{proof}
We use the dominated convergence theorem for stochastic integrals \citep[p.~133, property~j)]{chou80integrales}: if a sequence $(K^n)_{n\in\na}$ of predictable processes converges pointwise to a predictable process $K$, and $|K^n|\vee|K|\le G$ for some $G\in L(X)$, then $K^n\cdot X\to K\cdot X$ in $\T_{\mathrm{sm}}$.
Recall that for fixed $X$, stochastic integration defines a continuous linear map
\begin{equation}\label{eq:stochastic-integral-map}
(\mathcal L,\T_{\mathrm{ucp}})\ni H \mapsto H\cdot X\in(\mathcal{SM},\T_{\mathrm{sm}}).
\end{equation}

For the first assertion, let $\mathcal K$ be the $\T_{\mathrm{sm}}$-closure of $\{U\cdot X:U\in\mathcal E(\psi)\}$, and let $\mathcal C$ be the class of bounded predictable processes $H$ such that $H\cdot X\in\mathcal K$.
Common refinements of deterministic trading grids show that $\mathcal E(\psi)$ is a vector space; hence $\mathcal K$ and $\mathcal C$ are vector spaces.
Put $\mathcal E_{\mathrm b}\coloneqq\{H\in\mathcal E:\sup_{(\omega,t)\in\Omega\times\replus}|H_t(\omega)|<\infty\}$.
We first have $\mathcal E_{\mathrm b}\subseteq\mathcal C$.
Indeed, for $H\in\mathcal E_{\mathrm b}$, Proposition~\ref{prop:simple-process-approximation} gives $H^n\in\mathcal E(\psi)$ with $H^n\to H$ in ucp, and \eqref{eq:stochastic-integral-map} then gives $H^n\cdot X\to H\cdot X$ in $\T_{\mathrm{sm}}$.
The class $\mathcal C$ is stable under uniformly bounded pointwise convergence.
Indeed, if $H^n\in\mathcal C$, $|H^n|\le c$, and $H^n\to H$ pointwise, then $H$ is bounded and predictable, and the constant process $c$ belongs to $L(X)$.
Dominated convergence gives $H^n\cdot X\to H\cdot X$ in $\T_{\mathrm{sm}}$, so closedness of $\mathcal K$ yields $H\in\mathcal C$.

Moreover, $\indicatorset{[0,m]}\in\mathcal E_{\mathrm b}\subseteq\mathcal C$ and $\indicatorset{[0,m]}\to1$ pointwise, so $1\in\mathcal C$.
The class $\mathcal E_{\mathrm b}+\re$ is an algebra contained in $\mathcal C$ and generates $\mathcal P$: the truncated pointwise approximants in the proof of Proposition~\ref{prop:predictable-sigma-algebra} belong to $\mathcal E_{\mathrm b}$.
The functional monotone class theorem \citep[Theorem~7.4.1]{cohen15calculus} therefore shows that $\mathcal C$ contains every bounded predictable process.
Now let $V\in L(X)$ and put $V^{(m)}=V\indicatorset{\{|V|\le m\}}$.
Then $V^{(m)}$ is bounded and predictable, hence $V^{(m)}\cdot X\in\mathcal K$.
Since $V^{(m)}\to V$ pointwise and $|V^{(m)}|\le|V|$, dominated convergence gives $V^{(m)}\cdot X\to V\cdot X$ in $\T_{\mathrm{sm}}$.
Thus $V\cdot X\in\mathcal K$.
Metrizability of $\T_{\mathrm{sm}}$ now yields a sequence $(V^n)_{n\in\na}\subset\mathcal E(\psi)$ with $V^n\cdot X\to V\cdot X$ in $\T_{\mathrm{sm}}$.

For the second assertion, suppose in addition that $V\in\mathcal L$.
Simple predictable processes with almost surely finite coefficients are dense in $(\mathcal L,\T_{\mathrm{ucp}})$ \citep[Theorem~II.4.10]{protter05integration}.
Truncating their finitely many coefficients shows that $\mathcal S$ is also dense in $(\mathcal L,\T_{\mathrm{ucp}})$.
Together with Proposition~\ref{prop:simple-process-approximation}, this implies that $\mathcal S(\psi)$ is dense in $(\mathcal L,\T_{\mathrm{ucp}})$.
Hence there exists $(V^n)_{n\in\na}\subset\mathcal S(\psi)$ with $V^n\to V$ in ucp, and the continuity of \eqref{eq:stochastic-integral-map} gives $V^n\cdot X\to V\cdot X$ in $\T_{\mathrm{sm}}$ for the same sequence.
\end{proof}

\begin{remark}
Theorem~\ref{thm:ucp-approximation} is the main approximation result.
An element of $\mathcal{S}(\psi)$ is a trading strategy whose positions $\varphi$ are neural network functions of the trading time and finitely many stopped observations $(\tau, Y_{t_1}^\tau, \ldots, Y_{t_n}^\tau)$.
For a fixed càdlàg $\FF$-semimartingale price process $X$, the theorem shows that any strategy $V \in \mathcal{L}$ and its stochastic integral process $\int V\, \mathrm{d}X$ can both be approximated in $\T_{\mathrm{ucp}}$ and $\T_{\mathrm{sm}}$ by algorithmic strategies and their stochastic integrals, respectively.
If $V$ is predictable and $X$-integrable, then the stochastic integral $\int V\, \mathrm{d}X$ can still be approximated in $\T_{\mathrm{sm}}$ by stochastic integrals of algorithmic strategies in $\mathcal E(\psi)$.
In this sense algorithmic strategies form a dense approximation class for stochastic integrands and their integrals arising in mathematical finance.
A version of Theorem~\ref{thm:ucp-approximation} for predictable integrands and a stronger topology is provided below.
\end{remark}

For $p\in[1,\infty)$, write $\mathcal H_0^p=\{M\in\mathcal H^p:M_0=0\}$.
Let $\mathcal H_{\mathrm{sm}}^p$ consist of the real-valued special semimartingales $Z=N+C$ whose canonical decomposition has $N\in\mathcal H_0^p$ and a predictable finite variation part $C$ with $C_0=0$ and $\operatorname{Var}(C)_\infty\in L^p(\PP)$, where $\operatorname{Var}(C)$ denotes the total variation process.
Equip this space with the norm
\begin{equation*} \|Z\|_{\mathcal H_{\mathrm{sm}}^p}\coloneqq\|N\|_{\mathcal H^p}+\|\operatorname{Var}(C)_\infty\|_p. \end{equation*}
For $X\in\mathcal H_{\mathrm{sm}}^p$ with canonical decomposition $X=M+A$, put $B=\operatorname{Var}(A)$ and let $L^p(X)$ consist of the real-valued predictable processes $V$ for which
\begin{equation*} \|V\|_{L^p(X)}\coloneqq\big\|(V^2\cdot [M])_\infty^{1/2}\big\|_p+\big\|(|V|\cdot B)_\infty\big\|_p \end{equation*}
is finite.
The integrals in this definition are pathwise Lebesgue--Stieltjes integrals with respect to the increasing processes $[M]$ and $B$.
When $A=0$, $\|\cdot\|_{L^p(X)}=\|\cdot\|_{L^p(M)}$.
Every $V\in L^p(X)$ is $X$-integrable, with $V\cdot X\in\mathcal H_{\mathrm{sm}}^p$ and, for some constant $C_p<\infty$ depending only on $p$,
\begin{equation*}
\|V\cdot X\|_{\mathcal H_{\mathrm{sm}}^p}\le(C_p\vee1)\|V\|_{L^p(X)}.
\end{equation*}
Indeed, $V\cdot M\in\mathcal H_0^p$ with $[V\cdot M]=V^2\cdot [M]$ by \citet[Corollary~12.3.6]{cohen15calculus}, while $V\cdot A$ is predictable and satisfies $\operatorname{Var}(V\cdot A)_\infty\le(|V|\cdot B)_\infty$.

\begin{theorem}\label{thm:lp-integrand-approximation}
Let $p\in[1,\infty)$ and $X\in\mathcal H_{\mathrm{sm}}^p$, and suppose that $\psi$ satisfies Assumption~\ref{itm:sigmoidal-activation}.
Then $\mathcal E(\psi)$ is dense in $L^p(X)$ with respect to $\|\cdot\|_{L^p(X)}$.
In particular, every $U\in L^p(X)$ admits a sequence $(U^n)_{n\in\na}\subset\mathcal E(\psi)$ such that
\begin{equation*} \|U-U^n\|_{L^p(X)}+\|(U-U^n)\cdot X\|_{\mathcal H_{\mathrm{sm}}^p}\rightarrow0. \end{equation*}
\end{theorem}

\begin{proof}
Write $X=M+A$ and $B=\operatorname{Var}(A)$.
By the Burkholder--Davis--Gundy inequality \citep[Theorem~11.5.5]{cohen15calculus} and the assumption on $B$, we have $\smash{\mathbb E[[M]_\infty^{p/2}+B_\infty^p]<\infty}$.
Hence every bounded predictable process belongs to $L^p(X)$.
Fix $0\le s<t<\infty$ and a bounded $\Fcal_s$-measurable random variable $\xi$.
Since $\Fcal_s=\sigma(\mathcal N_\PP\cup\Fcal_s^0)$, we may replace $\xi$ by a bounded $\Fcal_s^0$-measurable version.
Set $w_{s,t}=([M]_t-[M]_s)^{p/2}+(B_t-B_s)^p$ and $\mu_{s,t}(D)=\mathbb E[w_{s,t}\indicatorset{D}]$ for $D\in\Fcal_s^0$.
Then $\mu_{s,t}$ is finite and $\mu_{s,t}\ll\PP$.
If $\mu_{s,t}=0$, take $\xi_n=0$.
Otherwise, apply Theorem~\ref{thm:orlicz-approximation} on $(\Omega,\Fcal_s^0,\mu_{s,t})$ with $\Phi(x)=|x|^p$ and $\Gamma_s=\operatorname{span}\{Y_r^j:0\le r\le s,\ 1\le j\le d\}$.
Since $\sigma(\Gamma_s)=\Fcal_s^0$ and the Orlicz norm associated with $\Phi$ is equivalent to the $L^p$-norm, there exist $\xi_n\in\mathcal{NN}_s(\psi)$ such that $\|\xi-\xi_n\|_{L^p(\mu_{s,t})}\to0$.
Moreover,
\begin{equation*} \|(\xi-\xi_n)\indicatorset{(s,t]}\|_{L^p(X)}\le2\|\xi-\xi_n\|_{L^p(\mu_{s,t})}\rightarrow0. \end{equation*}
Taking finite sums therefore approximates every bounded element of $\mathcal E$.

Let $\mathcal C$ be the bounded predictable processes in the $L^p(X)$-closure of $\mathcal E(\psi)$.
Since $\mathcal E(\psi)$ is a vector space, so is $\mathcal C$.
Let $V_n\in\mathcal C$ satisfy $|V_n|\le K<\infty$ and $V_n\to V$ pointwise.
Hence $|V|\le K$, and $\smash{\mathbb E[[M]_\infty^{p/2}+B_\infty^p]<\infty}$ implies that $[M]_\infty$ and $B_\infty$ are finite almost surely, so pathwise dominated convergence gives $((V_n-V)^2\cdot [M])_\infty^{p/2}\to0$ and $(|V_n-V|\cdot B)_\infty^p\to0$ almost surely.
The respective terms are bounded by $(2K)^p[M]_\infty^{p/2}$ and $(2K)^pB_\infty^p$.
The dominated convergence theorem therefore yields $\|V_n-V\|_{L^p(X)}\to0$, so $V\in\mathcal C$.
Since $\indicatorset{[0,n]}\in\mathcal E$ and $\indicatorset{[0,n]}\to1$ pointwise, the class $\mathcal C$ contains the constants.
The bounded elements of $\mathcal E$ are closed under multiplication and generate $\mathcal P$ by Proposition~\ref{prop:predictable-sigma-algebra}.
The functional monotone class theorem \citep[Theorem~7.4.1]{cohen15calculus} therefore implies that $\mathcal C$ contains every bounded predictable process.

For $U\in L^p(X)$, set $U^{(n)}=U\indicatorset{\{|U|\le n\}}$.
Then $((U-U^{(n)})^2\cdot [M])_\infty^{p/2}\to0$ and $(|U-U^{(n)}|\cdot B)_\infty^p\to0$ almost surely.
These terms are bounded by $(U^2\cdot [M])_\infty^{p/2}$ and $(|U|\cdot B)_\infty^p$, respectively.
Dominated convergence gives $\|U-U^{(n)}\|_{L^p(X)}\to0$.
Hence $\mathcal E(\psi)$ is dense in $L^p(X)$.
\end{proof}

\subsection{Mean-variance hedging under partial information}\label{subsec:mean-variance-hedging}
Theorem~\ref{thm:lp-integrand-approximation} gives density of $\mathcal E(\psi)$ in $L^2(X)$ for $X\in\mathcal H_{\mathrm{sm}}^2$.
We apply this approximation result to mean-variance hedging, a classical quadratic hedging criterion for incomplete markets \citep{rheinlaender97projections,schweizer01quadratic}.
In addition to restricting trading strategies to the algorithmic class, we allow the trader's information to be generated by a process whose filtration may be strictly smaller than $\FF$.
For such a filtration, we show that restricting trading strategies to the corresponding algorithmic class leaves the minimal mean-variance hedging error unchanged.

Let $X\in\mathcal H_{\mathrm{sm}}^2$, let $H\in L^2(\PP)$ be real-valued and $\Fcal_\infty$-measurable, and suppose that $\psi$ satisfies Assumption~\ref{itm:sigmoidal-activation}.
Let $Z$ be a jointly measurable, $\FF$-adapted process with values in $\re^m$, and set $\mathcal G_t^0=\sigma(Z_s:0\le s\le t)$, $\mathcal G_t=\sigma(\mathcal G_t^0\cup\mathcal N_\PP)$, and $\mathbb G=(\mathcal G_t)_{t\ge0}$.
Write $\mathcal P_{\mathbb G}$ for the predictable $\sigma$-algebra of $\mathbb G$ and set $L^2_{\mathbb G}(X)=\{U\in L^2(X):U\text{ is }\mathbb G\text{-predictable}\}$.
Let $\mathcal E_{\mathbb G}(\psi)$ denote the class $\mathcal E(\psi)$ with $Z$ in place of $Y$.
Thus, the strategies in $\mathcal E_{\mathbb G}(\psi)$ have deterministic trading times and neural network coefficients depending on finitely many observations of $Z$ up to the corresponding trading time.
All stochastic integrals below are defined with respect to $\FF$; in particular, neither right-continuity of $\mathbb G$ nor $\mathbb G$-adaptedness of $X$ is assumed.
For $c\in\re$ and $U\in L^2(X)$, set $J_X(c,U)=\mathbb E[(H-c-(U\cdot X)_\infty)^2]$.
The mean-variance hedging problem under $\mathbb G$ is \begin{equation}\label{eq:mv-problem}\inf_{\substack{c\in\re\\U\in L^2_{\mathbb G}(X)}}J_X(c,U).\end{equation}
Let $\mathcal W_{\mathbb G}(X)=\{c+(U\cdot X)_\infty:c\in\re,\ U\in L^2_{\mathbb G}(X)\}$, and let $\eta^\ast$ be the orthogonal projection of $H$ onto $\overline{\mathcal W_{\mathbb G}(X)}^{\,L^2(\PP)}$.
By the Hilbert space projection theorem, the value of~\eqref{eq:mv-problem} is $\|H-\eta^\ast\|_2^2$, and the infimum is attained if and only if $\eta^\ast\in\mathcal W_{\mathbb G}(X)$.
In particular, the infimum is attained whenever $\mathcal W_{\mathbb G}(X)$ is closed in $L^2(\PP)$.

Quadratic hedging under restricted information has been studied by \citet{schweizer94restricted} for risk minimization and by \citet{mania08partial} for mean-variance hedging, while \citet{ceci14restricted} develop Galtchouk--Kunita--Watanabe representations under restricted information.
Suppose that $X=M\in\mathcal H_0^2$, and define the finite measure $\nu$ on the $\FF$-predictable $\sigma$-algebra $\mathcal P$ by $\nu(D)=\mathbb E[\int_{(0,\infty)}\indicatorset{D}(\omega,t)\,\mathrm d[M]_t]$ for $D\in\mathcal P$.
After identifying integrands of zero $L^2(M)$-distance, the stochastic integral isometry identifies $L^2(M)$ with $L^2(\mathcal P,\nu)$.
The elements of $L^2(\mathcal P,\nu)$ admitting a $\mathcal P_{\mathbb G}$-measurable representative form a closed subspace; denote the corresponding orthogonal projection by $P_{\mathbb G}^M$.
Projecting $H-\mathbb E[H]$ onto the closed subspace of terminal stochastic integrals yields $U^\ast\in L^2(M)$ and $R\in L^2(\PP)$ such that $H=\mathbb E[H]+(U^\ast\cdot M)_\infty+R$, where $\mathbb E[R]=0$ and $\mathbb E[R(V\cdot M)_\infty]=0$ for every $V\in L^2(M)$.
Let $U^\ast_{\mathbb G}$ be a $\mathbb G$-predictable representative of $P_{\mathbb G}^M U^\ast$.
Orthogonality and the stochastic integral isometry give \begin{equation}\label{eq:mv-information}\inf_{\substack{c\in\re\\U\in L^2_{\mathbb G}(M)}}J_M(c,U)=\mathbb E[R^2]+\|U^\ast-U^\ast_{\mathbb G}\|_{L^2(\nu)}^2.\end{equation}
The infimum in~\eqref{eq:mv-information} is attained for $c=\mathbb E[H]$ and $U=U^\ast_{\mathbb G}$, and $U^\ast_{\mathbb G}$ is unique $\nu$-almost everywhere among minimizing integrands.
If $Z=Y$, then $\mathbb G=\FF$, so $U^\ast_{\mathbb G}=U^\ast$ $\nu$-almost everywhere and the second term on the right-hand side of~\eqref{eq:mv-information} vanishes.
The following result shows that restricting the integrand further to the algorithmic class preserves the value of the hedging problem and, in the martingale case, permits approximation of the minimizing integrand itself.

\begin{theorem}\label{thm:mean-variance-hedging}
Under the preceding assumptions, the following statements hold.
\begin{enumerate}[label=(\alph*)]
\item We have
\begin{equation}\label{eq:mv-optimal-value}\inf_{\substack{c\in\re\\U\in\mathcal E_{\mathbb G}(\psi)}}J_X(c,U)=\inf_{\substack{c\in\re\\U\in L^2_{\mathbb G}(X)}}J_X(c,U).\end{equation}
If $(c^\ast,U^\ast)$ attains the infimum on the right-hand side of~\eqref{eq:mv-optimal-value}, then for every $\delta>0$, there exists $U\in\mathcal E_{\mathbb G}(\psi)$ such that $J_X(c^\ast,U)<J_X(c^\ast,U^\ast)+\delta$.
\item If $X=M\in\mathcal H_0^2$, then there exists a sequence $(U^n)_{n\in\na}\subset\mathcal E_{\mathbb G}(\psi)$ such that
\begin{equation*}\|U^n-U^\ast_{\mathbb G}\|_{L^2(M)}+\|(U^n-U^\ast_{\mathbb G})\cdot M\|_{\mathcal H^2}\rightarrow0.\end{equation*} 
Consequently, $J_M(\mathbb E[H],U^n)\to J_M(\mathbb E[H],U^\ast_{\mathbb G})$, and the latter equals the infimum in~\eqref{eq:mv-information}.
\end{enumerate}
\end{theorem}

\begin{proof}
We first adapt the density argument of Theorem~\ref{thm:lp-integrand-approximation} to $\mathbb G$.
Since $Z$ is $\FF$-adapted and $\Fcal_t$ contains $\mathcal N_\PP$, we have $\mathcal G_t\subseteq\Fcal_t$, hence $\mathcal P_{\mathbb G}\subseteq\mathcal P$.
Write $X=M+A$ and $B=\operatorname{Var}(A)$.
By assumption, $\mathbb E[[M]_\infty+B_\infty^2]<\infty$, so every bounded $\mathbb G$-predictable process belongs to $L^2(X)$.
In particular, $\mathcal E_{\mathbb G}(\psi)\subseteq L^2_{\mathbb G}(X)$, because $\psi$ is bounded.
Note that by construction, $\mathcal E_{\mathbb G}(\psi)$ is a vector space.

Fix $0\le s<t<\infty$ and a bounded $\mathcal G_s$-measurable random variable $\xi$.
Since $\mathcal G_s$ is obtained by adjoining $\PP$-null sets to $\mathcal G_s^0$, the variable $\xi$ has a bounded $\mathcal G_s^0$-measurable version, which we use below.
Set
\begin{equation*}
w_{s,t}=[M]_t-[M]_s+(B_t-B_s)^2,\qquad
\mu_{s,t}(D)=\mathbb E[w_{s,t}\indicatorset{D}],\quad D\in\mathcal G_s^0.
\end{equation*}
This is a finite measure with $\mu_{s,t}\ll\PP$, so replacing $\xi$ by its chosen $\mathcal G_s^0$-measurable version does not change the relevant integrals.
If $\mu_{s,t}=0$, take $\xi_n=0$.
Otherwise, apply Theorem~\ref{thm:orlicz-approximation} with $\Phi(x)=x^2$ and
\begin{equation*}
\Gamma_s=\operatorname{span}\{Z_r^j:0\le r\le s,\ 1\le j\le m\}.
\end{equation*}
This family is closed under addition and subtraction and generates $\mathcal G_s^0$; moreover, $\Phi\in\Delta_2$ and its Orlicz norm is equivalent to the $L^2$-norm.
We obtain neural network coefficients $\xi_n$, each using finitely many observations up to $s$ such that $\|\xi-\xi_n\|_{L^2(\mu_{s,t})}\to0$.
For $\zeta_n=\xi-\xi_n$,
\begin{equation*}
\|\zeta_n\indicatorset{(s,t]}\|_{L^2(X)}=\bigl(\mathbb E[\zeta_n^2([M]_t-[M]_s)]\bigr)^{1/2}+\bigl(\mathbb E[\zeta_n^2(B_t-B_s)^2]\bigr)^{1/2}\le2\|\zeta_n\|_{L^2(\mu_{s,t})}\rightarrow0.
\end{equation*}
Finite sums therefore approximate every bounded elementary $\mathbb G$-predictable process with deterministic times.
Terms supported at time zero have zero $L^2(X)$-seminorm, since $M_0=A_0=0$.

Let $\mathcal C$ be the bounded $\mathbb G$-predictable processes that can be approximated in $L^2(X)$ by elements of $\mathcal E_{\mathbb G}(\psi)$.
This is a vector space.
If $V_n\in\mathcal C$, $|V_n|\le K$, and $V_n\to V$ pointwise, pathwise dominated convergence gives
\begin{equation*}
((V_n-V)^2\cdot [M])_\infty\rightarrow0,
\qquad
(|V_n-V|\cdot B)_\infty^2\rightarrow0
\end{equation*}
almost surely.
The respective terms are bounded by $4K^2[M]_\infty$ and $4K^2B_\infty^2$.
Dominated convergence under $\PP$ yields $\|V_n-V\|_{L^2(X)}\to0$.
Since $V$ is bounded and $\mathbb G$-predictable, and each $V_n$ lies in the $L^2(X)$-closure of $\mathcal E_{\mathbb G}(\psi)$, the triangle inequality shows that $V$ lies in this closure as well, hence $V\in\mathcal C$.
In particular, $\indicatorset{[0,n]}\to1$ shows that $\mathcal C$ contains the constants.
Bounded elementary $\mathbb G$-predictable processes are closed under multiplication and, together with constants, form an algebra generating $\mathcal P_{\mathbb G}$.
Indeed, the left-endpoint approximations of left-continuous adapted processes used in Proposition~\ref{prop:predictable-sigma-algebra} apply to any filtration.
The functional monotone class theorem now implies that $\mathcal C$ contains every bounded $\mathbb G$-predictable process.

For $U\in L^2_{\mathbb G}(X)$, put $U^{(n)}=U\indicatorset{\{|U|\le n\}}$.
The random variables $((U-U^{(n)})^2\cdot [M])_\infty$ and $(|U-U^{(n)}|\cdot B)_\infty^2$ tend to zero almost surely and are bounded by $(U^2\cdot [M])_\infty$ and $(|U|\cdot B)_\infty^2$, respectively.
Dominated convergence therefore gives $\|U-U^{(n)}\|_{L^2(X)}\to0$.
Choose $V_n\in\mathcal E_{\mathbb G}(\psi)$ with $\|V_n-U^{(n)}\|_{L^2(X)}<1/n$.
Then $\|V_n-U\|_{L^2(X)}\to0$.
This proves density within $L^2_{\mathbb G}(X)$.

For $U\in L^2(X)$, the integral $U\cdot M$ has a terminal value in $L^2(\PP)$, and $U\cdot A$ converges in $L^2(\PP)$ because its total variation is bounded by $(|U|\cdot B)_\infty\in L^2(\PP)$.
The terminal integral map is linear and continuous, with
\begin{equation*}
\|(U\cdot X)_\infty\|_2\le\|(U\cdot M)_\infty\|_2+\|(|U|\cdot B)_\infty\|_2=\|U\|_{L^2(X)}.
\end{equation*}
It follows from density and $\mathcal E_{\mathbb G}(\psi)\subseteq L^2_{\mathbb G}(X)$ that terminal wealths from $\mathcal E_{\mathbb G}(\psi)$ have the same $L^2(\PP)$-closure as $\mathcal W_{\mathbb G}(X)$.
Since the two classes of terminal wealths have the same $L^2(\PP)$-closure and $L^2(\PP)\ni\eta\mapsto\|H-\eta\|_2^2$ is continuous, their infimal distances from $H$ coincide.
Orthogonal projection onto this common closed linear space gives
\begin{equation*}
\|H-\eta\|_2^2=\|H-\eta^\ast\|_2^2+\|\eta-\eta^\ast\|_2^2,\qquad\eta\in\overline{\mathcal W_{\mathbb G}(X)}^{\,L^2(\PP)},
\end{equation*}
so the common infimum equals $\|H-\eta^\ast\|_2^2$. This proves~\eqref{eq:mv-optimal-value}.
If $(c^\ast,U^\ast)$ is optimal, then $c^\ast+(U^\ast\cdot X)_\infty=\eta^\ast$.
For $U\in\mathcal E_{\mathbb G}(\psi)$, the same identity yields
\begin{equation*}
J_X(c^\ast,U)-J_X(c^\ast,U^\ast)=\|((U-U^\ast)\cdot X)_\infty\|_2^2\le\|U-U^\ast\|_{L^2(X)}^2.
\end{equation*}
Approximation within distance $\sqrt\delta$ proves the last assertion of~(a).

For~(b), the stochastic integral isometry
\begin{equation*}
\|(V\cdot M)_\infty\|_2^2
=\mathbb E[(V^2\cdot [M])_\infty]
=\|V\|_{L^2(\nu)}^2
\end{equation*}
identifies $L^2(\mathcal P,\nu)$ with a closed linear subspace of centered random variables in $L^2(\PP)$.
Projecting $H-\mathbb E[H]$ onto this subspace gives $U^\ast\in L^2(M)$ and a centered remainder $R\in L^2(\PP)$ orthogonal to all terminal stochastic integrals, yielding the stated decomposition.
The elements admitting a $\mathcal P_{\mathbb G}$-measurable representative also form a closed subspace of $L^2(\mathcal P,\nu)$.
To see this, take $\mathcal P_{\mathbb G}$-measurable representatives of an $L^2(\nu)$-convergent sequence and extract a subsequence converging $\nu$-almost everywhere.
Its limit on the set on which the subsequence converges to a finite limit, extended by zero elsewhere, is a $\mathcal P_{\mathbb G}$-measurable representative of the $L^2$-limit.
Thus $P_{\mathbb G}^M$ exists.
For related Galtchouk--Kunita--Watanabe decompositions under restricted information, see \citet{ceci14restricted}.

For $c\in\re$ and $U\in L^2_{\mathbb G}(M)$, write
\begin{equation*}
H-c-(U\cdot M)_\infty=R+((U^\ast-U^\ast_{\mathbb G})\cdot M)_\infty+((U^\ast_{\mathbb G}-U)\cdot M)_\infty+\mathbb E[H]-c.
\end{equation*}
The first term is orthogonal to both stochastic integrals by construction.
The two integrals are orthogonal by the isometry and the defining property of $P_{\mathbb G}^M$.
All three random terms are centered, so they are orthogonal to the constant term.
Orthogonality yields~\eqref{eq:mv-information}, including attainment and uniqueness.
Choose $U^n\in\mathcal E_{\mathbb G}(\psi)$ with $\|U^n-U^\ast_{\mathbb G}\|_{L^2(M)}\to0$, as proved above.
Doob's inequality gives
\begin{equation*}
\|(U^n-U^\ast_{\mathbb G})\cdot M\|_{\mathcal H^2}
\le2\|U^n-U^\ast_{\mathbb G}\|_{L^2(M)}\rightarrow0,
\end{equation*}
and continuity of the terminal integral map gives convergence of the hedging errors.
Finally, $Z=Y$ implies $\mathbb G=\FF$, so $U^\ast_{\mathbb G}=U^\ast$ in $L^2(\nu)$.
\end{proof}

\subsection{No free lunch for algorithmic strategies}\label{subsec:no-free-lunch}
Theorem~\ref{thm:mean-variance-hedging}(a) shows that restricting integrands in $L_\mathbb{G}^2(X)$ to algorithmic strategies does not change the minimal mean-variance hedging error.
We next show that the same restriction preserves the classical relation between trading opportunities and martingale measures.
For bounded elementary predictable processes, suitable no-free-lunch conditions are characterized by the existence of an equivalent martingale measure.
Since algorithmic strategies form a smaller class, the absence of free lunches generated by them is a priori a weaker requirement.

Throughout this subsection, retain Standing Assumption~\ref{ass:levy-driver} and suppose that $\psi$ satisfies Assumption~\ref{itm:sigmoidal-activation}.
Fix $p\in[1,\infty)$ and $q\in(1,\infty]$ with $1/p+1/q=1$.
Let $X=(X_t)_{t\in[0,1]}$ be an $\FF$-adapted c\`{a}dl\`{a}g $\re^d$-valued process such that $X_t\in L^p(\PP)$ for every $t\in[0,1]$.
No semimartingale assumption is imposed on $X$.
Let $\mathcal K_\psi(X)$ denote the set of terminal gains from zero initial wealth of the form
\begin{equation}\label{eq:algo-gain}
G=\sum_{i=1}^n\sum_{j=1}^d\varphi_i^j\bigl(X_{t_{i+1}}^j-X_{t_i}^j\bigr).
\end{equation}
Here $n\in\na$, the times $0=t_1<\cdots<t_{n+1}=1$ are deterministic, and $\varphi_i^j\in\mathcal{NN}_{t_i}(\psi)$ for $i=1,\ldots,n$ and $j=1,\ldots,d$, as in Definition~\ref{def:neural-random-variables}.
Thus the initial term in the elementary integral convention is omitted.
Since $\psi$ is bounded, $\mathcal K_\psi(X)\subseteq L^p(\PP)$.
To include free lunches arising as $L^p$-limits of claims dominated by such gains, set
\begin{equation*}
\mathcal C_\psi(X)\coloneqq
\overline{\mathcal K_\psi(X)-\mathcal B_1^+}^{\,L^p(\PP)},
\end{equation*}
where $\mathcal B_1^+\coloneqq\{f\in L^\infty(\Omega,\Fcal_1,\PP):f\ge0\ \PP\text{-a.s.}\}$.
Specifically, we establish that $\mathcal C_\psi(X)\cap L^p_+(\PP)=\{0\}$ is equivalent to the existence of a probability measure $\QQ\sim\PP$ with $\mathrm d\QQ/\mathrm d\PP\in L^q(\PP)$ under which $X$ is an $\FF$-martingale.

The classical martingale measure characterization is obtained by separating attainable gains from the positive cone $L^p_+(\PP)$.
The $L^1$ separation theorem of \citet[Theorem~2]{yan80caracterisation}, extended to $L^p$ by \citet[Theorem~1]{ansel90remarques}, produces strictly positive $L^q$ functionals under the corresponding no-free-lunch conditions.
Let $\mathcal K_{\mathrm e}(X)$ denote the terminal gains in~\eqref{eq:algo-gain} obtained by replacing $\varphi_i^j\in\mathcal{NN}_{t_i}(\psi)$ with $\varphi_i^j\in\mathcal L^\infty(\Fcal_{t_i},\PP)$.
For this classical space of gains, \citet[Theorems~2 and~3]{stricker90arbitrage} give the relevant martingale measure characterizations.
The connection with algorithmic strategies follows from Theorem~\ref{thm:orlicz-approximation}: each bounded $\Fcal_{t_i}$-measurable coefficient can be approximated in the $L^p$-norm weighted by the corresponding price increment by elements of $\mathcal{NN}_{t_i}(\psi)$.
Consequently, $\mathcal K_\psi(X)$ and $\mathcal K_{\mathrm e}(X)$ have the same $L^p(\PP)$-closure, and the classical separation conditions can be transferred to the restricted class of gains.

\begin{theorem}\label{thm:equivalent-martingale-measures}
The following conditions are equivalent.
\begin{enumerate}[label=(\alph*)]
\item
There exists a probability measure $\QQ\sim\PP$ on $(\Omega,\Fcal)$ with $\mathrm d\QQ/\mathrm d\PP\in L^q(\PP)$ such that $X$ is an $\FF$-martingale under $\QQ$.
\item $\mathcal C_\psi(X)\cap L^p_+(\PP)=\{0\}$.
\end{enumerate}
If $X$ is continuous, conditions~\textup{(a)} and~\textup{(b)} are also equivalent to
\begin{enumerate}[label=(\alph*),resume]
\item $\overline{\mathcal K_\psi(X)}^{\,L^p(\PP)}\cap L^p_+(\PP)=\{0\}$.
\end{enumerate}
\end{theorem}

\begin{proof}
Let $\mathcal K_{\mathrm e}(X)$ denote the corresponding space of terminal gains obtained by allowing arbitrary bounded $\Fcal_{t_i}$-measurable coefficients.
Since $\psi$ is bounded,
$\mathcal K_\psi(X)\subseteq\mathcal K_{\mathrm e}(X)\subseteq L^p(\PP)$.
We first show that
\begin{equation}\label{eq:gain-closure-equality}
\overline{\mathcal K_\psi(X)}^{\,L^p(\PP)}=\overline{\mathcal K_{\mathrm e}(X)}^{\,L^p(\PP)}.
\end{equation}

Fix $0\le s<t\le1$, $j\in\{1,\ldots,d\}$, and a bounded $\Fcal_s$-measurable random variable $\eta$.
Choose a bounded $\Fcal_s^0$-measurable version of $\eta$ and define the finite measure $\mu_{s,t}^j(A)\coloneqq\mathbb E[\indicatorset{A}|X_t^j-X_s^j|^p]$, $A\in\Fcal_s^0$.
Let $\Gamma_s\coloneqq \operatorname{span}\{Y_r^k:0\le r\le s,\ 1\le k\le d\}$, which generates $\Fcal_s^0$.
If $\mu_{s,t}^j\neq0$, Theorem~\ref{thm:orlicz-approximation}, applied with $\Phi(u)=|u|^p$, yields a sequence $\eta_m\in\mathcal{NN}(\Gamma_s,\psi)$ converging to $\eta$ in $L^p(\mu_{s,t}^j)$.
Each $\eta_m$ involves only finitely many elements of $\Gamma_s$, hence $\eta_m\in\mathcal{NN}_s(\psi)$.
If $\mu_{s,t}^j=0$, take $\eta_m=0$.
In either case,
\begin{equation*}
\|(\eta_m-\eta)(X_t^j-X_s^j)\|_{L^p(\PP)}^p=\int_\Omega|\eta_m-\eta|^p\,\mathrm d\mu_{s,t}^j\rightarrow0.
\end{equation*}
Applying this approximation to the finitely many coefficients of any $G\in\mathcal K_{\mathrm e}(X)$ and using the triangle inequality gives $G_m\in\mathcal K_\psi(X)$ with $G_m\to G$ in $L^p(\PP)$.
Together with $\mathcal K_\psi(X)\subseteq\mathcal K_{\mathrm e}(X)$, this proves~\eqref{eq:gain-closure-equality}.
Using $\mathcal K_\psi(X)\subseteq\mathcal K_{\mathrm e}(X)$ yields
\begin{equation}\label{eq:cone-closure-equality}
\mathcal C_\psi(X)=\overline{\mathcal K_{\mathrm e}(X)-\mathcal B_1^+}^{\,L^p(\PP)}.
\end{equation}
All random variables in~\eqref{eq:gain-closure-equality} and \eqref{eq:cone-closure-equality} are $\Fcal_1$-measurable.
Since $L^p(\Omega,\Fcal_1,\PP)$ is a closed subspace of $L^p(\PP)$, the same closures are obtained in $L^p(\Omega,\Fcal_1,\PP)$.

Condition~\textup{(a)} is equivalent to Stricker's property $\mathcal M^p$ on $(\Omega,\Fcal_1,(\Fcal_t)_{0\le t\le1},\PP)$.
Indeed, if $\QQ$ satisfies~\textup{(a)} with density $Z$, then $\mathbb E[Z\mid\Fcal_1]\in L^q(\PP)$ is strictly positive almost surely and is the density of $\QQ|_{\Fcal_1}$.
Conversely, an $\Fcal_1$-measurable density $Z\in L^q(\PP)$ of an equivalent martingale measure on $\Fcal_1$ defines such a measure on $\Fcal$ by $\QQ(A)=\mathbb E[Z\indicatorset{A}]$, $A\in\Fcal$.
The martingale property on $[0,1]$ is unchanged.

The space $\mathcal K_{\mathrm e}(X)$ coincides with the space of terminal gains from bounded elementary predictable strategies used by \citet{stricker90arbitrage}, up to the immaterial insertion of the endpoints $0$ and $1$ with zero coefficients.
Hence \citet[Theorem~2]{stricker90arbitrage}, together with \eqref{eq:cone-closure-equality}, shows that property $\mathcal M^p$ is equivalent to $\mathcal C_\psi(X)\cap L^p_+(\PP)=\{0\}$.
This proves the equivalence of~\textup{(a)} and~\textup{(b)}.

If $X$ is continuous, \citet[Theorem~3]{stricker90arbitrage} states that property $\mathcal M^p$ is equivalent to
\begin{equation*}
\overline{\mathcal K_{\mathrm e}(X)}^{\,L^p(\Omega,\Fcal_1,\PP)}\cap L^p_+(\Omega,\Fcal_1,\PP)=\{0\}.
\end{equation*}
By~\eqref{eq:gain-closure-equality}, this is equivalent to $\overline{\mathcal K_\psi(X)}^{\,L^p(\PP)}\cap L^p_+(\PP)=\{0\}$, which is condition~\textup{(c)}.
\end{proof}

\subsection{A Bichteler--Dellacherie characterization of semimartingales}\label{subsec:semimartingale-characterization}
The semimartingale property is fundamental to the results on stochastic integration developed above.
For a semimartingale $X$, stochastic integration maps $(\mathcal L,\T_{\mathrm{ucp}})$ continuously into $(\mathcal{SM},\T_{\mathrm{sm}})$.
We next characterize the semimartingale property of $X$ in terms of the terminal integrals generated by algorithmic strategies.
Such a characterization shows that algorithmic strategies are not only dense in spaces of stochastic integrands, but also suffice to characterize the class of integrators for which the preceding approximation theory applies.

Throughout this subsection, retain Standing Assumption~\ref{ass:levy-driver} and suppose that $\psi$ satisfies Assumption~\ref{itm:sigmoidal-activation}.
Fix $X\in\mathcal D$ and $T>0$, without assuming that $X$ is a semimartingale.
The Bichteler--Dellacherie theorem tests the semimartingale property through simple predictable integrands that are uniformly bounded by one.
To obtain the corresponding algorithmic class, set $\sigma=\tanh$ and define $\mathcal S(\psi;\sigma)\coloneqq\{\sigma\circ V:V\in\mathcal S(\psi)\}$, where $(\sigma\circ V)_t=\sigma(V_t)$.
Thus every process in this class takes values in $[-1,1]$, while its coefficient at each stopping time remains determined by a neural network of stopped observations.
The corresponding algorithmic criterion is boundedness in probability of
\begin{equation*}
\{I(V,X)_T:V\in\mathcal S(\psi;\sigma)\}.
\end{equation*}

Let $\mathcal S_{T,1}$ denote the simple predictable processes on $[0,T]$ that are uniformly bounded by one and whose initial position is zero.
Then $X$ is an $\FF$-semimartingale on $[0,T]$ if and only if $\{I(H,X)_T:H\in\mathcal S_{T,1}\}$ is bounded in probability \citep[Theorem~BD]{beiglbock14riemann}.
To pass from this classical criterion to algorithmic strategies, it suffices to recover the terminal integrals generated by $\mathcal S_{T,1}$ from those generated by $\mathcal S(\psi;\sigma)$.
The following theorem yields the classical characterization of semimartingales using only algorithmic strategies.

\begin{theorem}\label{thm:semimartingale-characterization}
Under the preceding assumptions, the following conditions are equivalent.
\begin{enumerate}[label=(\alph*)]
\item $X$ is an $\FF$-semimartingale on $[0,T]$.
\item The family $\{I(V,X)_T:V\in\mathcal S(\psi;\sigma)\}$ is bounded in probability.
\end{enumerate}
\end{theorem}

\begin{proof}
Let $\mathcal S_{T,1}$ denote the simple predictable processes on $[0,T]$ with zero initial position and uniformly bounded by one.
By the Bichteler--Dellacherie theorem, $X$ is an $\FF$-semimartingale on $[0,T]$ if and only if the family
\begin{equation}\label{eq:bd-test-family}
\{I(H,X)_T:H\in\mathcal S_{T,1}\}
\end{equation}
is bounded in probability; see \citet[Theorem~BD]{beiglbock14riemann}.
The required hypotheses on $\FF$ follow from Standing Assumption~\ref{ass:levy-driver}.

Assume first that \textup{(a)} holds and fix $V\in\mathcal S(\psi;\sigma)$.
Write $V=\sigma\circ U$ with $U\in\mathcal S(\psi)$.
Since $\sigma(0)=0$ and $|\sigma|\le1$, the process $H\coloneqq V\indicatorset{(0,T]}$ belongs to $\mathcal S_{T,1}$.
Moreover, $I(V,X)_T=V_0X_0+I(H,X)_T$ and $|V_0X_0|\le|X_0|$.
Hence the Bichteler--Dellacherie theorem implies that $\{I(V,X)_T:V\in\mathcal S(\psi;\sigma)\}$ is bounded in probability.
Thus \textup{(a)} implies \textup{(b)}.

Conversely, assume \textup{(b)} and set $\mathcal A_T\coloneqq\{I(V,X)_T:V\in\mathcal S(\psi;\sigma)\}$.
Let $\overline{\mathcal A_T}$ denote the closure of $\mathcal A_T$ under convergence in probability.
We show that every terminal integral in~\eqref{eq:bd-test-family} belongs to $\overline{\mathcal A_T}$.
Fix $H\in\mathcal S_{T,1}$.
Choose a representation
\begin{equation*}
H=\sum_{i=1}^n\xi_i\indicatorset{(\tau_i,\tau_{i+1}]}(\cdot),
\end{equation*}
where $0=\tau_1\le\cdots\le\tau_{n+1}=T$ are stopping times and each $\xi_i$ is bounded and $\Fcal_{\tau_i}$-measurable.
Since $|H|\le1$, the coefficients may be chosen with values in $[-1,1]$.

For $m\ge2$, set $a_m=1-m^{-1}$ and, for $i=1,\ldots,n$, define $g_{i,m}\coloneqq\operatorname{arctanh}(a_m\xi_i)$.
Each $g_{i,m}$ is bounded and $\Fcal_{\tau_i}$-measurable.
Applying Theorem~\ref{thm:random-variable-approximation}\textup{(b)(iii)} with $\Phi(x)=x^2$, choose $f_{i,m}\in\mathcal{NN}_{\tau_i}(\psi)$ such that
\begin{equation*}
\|f_{i,m}-g_{i,m}\|_{L^2(\PP)}<m^{-1}.
\end{equation*}
Since $\sigma=\tanh$ is $1$-Lipschitz and $\sigma(g_{i,m})=a_m\xi_i$, we obtain
\begin{equation}\label{eq:stopping-time-coefficient-approximation}
\|\sigma(f_{i,m})-\xi_i\|_{L^2(\PP)}\le\|f_{i,m}-g_{i,m}\|_{L^2(\PP)}+(1-a_m)\|\xi_i\|_{L^2(\PP)}\le\frac{2}{m}.
\end{equation}

Define $U^m\coloneqq\sum_{i=1}^n f_{i,m}\indicatorset{(\tau_i,\tau_{i+1}]}(\cdot)$ and set $V^m\coloneqq\sigma\circ U^m$.
Then $V^m\in\mathcal S(\psi;\sigma)$.
By the elementary integral formula,
\begin{equation}\label{eq:simple-integral-approximation}
I(V^m,X)_T-I(H,X)_T
=
\sum_{i=1}^n
\bigl(\sigma(f_{i,m})-\xi_i\bigr)
\bigl(X_{\tau_{i+1}}-X_{\tau_i}\bigr).
\end{equation}
By~\eqref{eq:stopping-time-coefficient-approximation}, $\sigma(f_{i,m})-\xi_i\to0$ in probability for every $i$.
Since $X_{\tau_{i+1}}-X_{\tau_i}$ is finite almost surely, multiplication by this random variable preserves convergence in probability.
Since the sum in~\eqref{eq:simple-integral-approximation} contains only finitely many terms, it follows that
\begin{equation*}
I(V^m,X)_T\longrightarrow I(H,X)_T
\qquad\text{in probability}.
\end{equation*}
Consequently,
\begin{equation}\label{eq:bd-family-in-closure}
\{I(H,X)_T:H\in\mathcal S_{T,1}\}
\subseteq
\overline{\mathcal A_T}.
\end{equation}

It remains to show that boundedness in probability passes from $\mathcal A_T$ to $\overline{\mathcal A_T}$.
Fix $Z\in\overline{\mathcal A_T}$ and choose $Z_j\in\mathcal A_T$ such that $Z_j\to Z$ in probability.
For every $R>0$ and every $j\in\na$,
\begin{equation*}
\PP(|Z|>R)\le\PP(|Z-Z_j|>R/2)+\PP(|Z_j|>R/2).
\end{equation*}
Letting $j\to\infty$ yields
\begin{equation*}
\PP(|Z|>R)\le\sup_{W\in\mathcal A_T}\PP(|W|>R/2).
\end{equation*}
Taking the supremum over $Z\in\overline{\mathcal A_T}$ and then letting $R\to\infty$ shows that $\overline{\mathcal A_T}$ is bounded in probability.
By~\eqref{eq:bd-family-in-closure}, the family in~\eqref{eq:bd-test-family} is therefore bounded in probability.
The Bichteler--Dellacherie theorem implies that $X$ is an $\FF$-semimartingale on $[0,T]$.
Thus \textup{(b)} implies \textup{(a)}.
\end{proof}

\section{Conclusion}\label{sec:conclusion}
We have shown that restricting the coefficients of simple predictable processes to neural network functions preserves several classical results of stochastic integration and mathematical finance.
The resulting class is dense in spaces of predictable integrands, and the corresponding approximation extends to stochastic integrals.
It yields a Bichteler--Dellacherie characterization of semimartingales, approximation results for mean-variance hedging under partial information, and a no-free-lunch characterization in terms of equivalent martingale measures.
These results rest on a universal approximation theorem for neural networks in Orlicz spaces over general measurable spaces and on corresponding approximation results for measurable random variables.

The results on stochastic integration are developed under a filtration generated by a L\'evy process, and the approximation theorems are qualitative.
Natural extensions are to identify more general filtrations for which the same results hold and to obtain quantitative approximation rates and bounds on network complexity.
In mathematical finance, a further direction is to determine which results persist under portfolio constraints and transaction costs.
More generally, it remains to determine which other structured subclasses of elementary predictable processes suffice to recover the classical theory of stochastic integration.

\bibliographystyle{abbrvnat}
\bibliography{ctas-main}

\begin{thebibliography}{52}
\providecommand{\natexlab}[1]{#1}
\providecommand{\url}[1]{\texttt{#1}}
\expandafter\ifx\csname urlstyle\endcsname\relax
  \providecommand{\doi}[1]{doi: #1}\else
  \providecommand{\doi}{doi: \begingroup \urlstyle{rm}\Url}\fi

\bibitem[Ansel and Stricker(1990)]{ansel90remarques}
J.-P. Ansel and C.~Stricker.
\newblock Quelques remarques sur un th{\'e}or{\`e}me de {Yan}.
\newblock \emph{S{\'e}minaire de Probabilit{\'e}s}, 24:\penalty0 266--274, 1990.

\bibitem[Ansel and Stricker(1993)]{ansel93decomposition}
J.-P. Ansel and C.~Stricker.
\newblock D{\'e}composition de {Kunita--Watanabe}.
\newblock \emph{S{\'e}minaire de Probabilit{\'e}s}, 27:\penalty0 30--32, 1993.

\bibitem[Arandjelovi\'{c}(2024)]{arandjelovic24theory}
A.~Arandjelovi\'{c}.
\newblock \emph{Theory of Neural Networks with Applications in Finance, Insurance and Climate-Economy Modelling}.
\newblock PhD thesis, TU Wien and Macquarie University, 2024.

\bibitem[Arandjelovi\'{c} et~al.(2025)Arandjelovi\'{c}, Rheinl{\"a}nder, and Shevchenko]{arandjelovic25importance}
A.~Arandjelovi\'{c}, T.~Rheinl{\"a}nder, and P.~V. Shevchenko.
\newblock Importance sampling for option pricing with feedforward neural networks.
\newblock \emph{Finance and Stochastics}, 29:\penalty0 97--141, 2025.

\bibitem[Bach(2017)]{bach17convex}
F.~Bach.
\newblock Breaking the curse of dimensionality with convex neural networks.
\newblock \emph{Journal of Machine Learning Research}, 18\penalty0 (19):\penalty0 1--53, 2017.

\bibitem[Beiglb{\"o}ck and Siorpaes(2014)]{beiglbock14riemann}
M.~Beiglb{\"o}ck and P.~Siorpaes.
\newblock Riemann-integration and a new proof of the {Bichteler--Dellacherie} theorem.
\newblock \emph{Stochastic Processes and their Applications}, 124\penalty0 (3):\penalty0 1226--1235, 2014.

\bibitem[Bichteler(1979)]{bichteler79integrators}
K.~Bichteler.
\newblock Stochastic integrators.
\newblock \emph{Bulletin of the American Mathematical Society}, 1\penalty0 (5):\penalty0 761--765, 1979.

\bibitem[Bichteler(1981)]{bichteler81integration}
K.~Bichteler.
\newblock Stochastic integration and {$L^{p}$}-theory of semimartingales.
\newblock \emph{The Annals of Probability}, 9\penalty0 (1):\penalty0 49--89, 1981.

\bibitem[Billingsley(2012)]{billingsley12probability}
P.~Billingsley.
\newblock \emph{Probability and Measure}.
\newblock John Wiley \& Sons, 4th edition, 2012.

\bibitem[Bogachev(2007{\natexlab{a}})]{bogachev07measure1}
V.~I. Bogachev.
\newblock \emph{Measure Theory, Volume 1}.
\newblock Springer, 2007{\natexlab{a}}.

\bibitem[Bogachev(2007{\natexlab{b}})]{bogachev07measure2}
V.~I. Bogachev.
\newblock \emph{Measure Theory, Volume 2}.
\newblock Springer, 2007{\natexlab{b}}.

\bibitem[Buehler et~al.(2019)Buehler, Gonon, Teichmann, and Wood]{buehler19hedging}
H.~Buehler, L.~Gonon, J.~Teichmann, and B.~Wood.
\newblock Deep hedging.
\newblock \emph{Quantitative Finance}, 19\penalty0 (8):\penalty0 1271--1291, 2019.

\bibitem[Cauchy(1847)]{cauchy47methode}
A.-L. Cauchy.
\newblock M{\'e}thode g{\'e}n{\'e}rale pour la r{\'e}solution des syst{\`e}mes d'{\'e}quations simultan{\'e}es.
\newblock \emph{Comptes rendus hebdomadaires des s{\'e}ances de l'Acad{\'e}mie des sciences}, 25:\penalty0 536--538, 1847.

\bibitem[Ceci et~al.(2014)Ceci, Cretarola, and Russo]{ceci14restricted}
C.~Ceci, A.~Cretarola, and F.~Russo.
\newblock {GKW} representation theorem under restricted information: {An} application to risk-minimization.
\newblock \emph{Stochastics and Dynamics}, 14\penalty0 (2):\penalty0 1350019, 2014.

\bibitem[Cheridito(2003)]{cheridito03fractional}
P.~Cheridito.
\newblock Arbitrage in fractional {Brownian} motion models.
\newblock \emph{Finance and Stochastics}, 7:\penalty0 533--553, 2003.

\bibitem[Chou et~al.(1980)Chou, Meyer, and Stricker]{chou80integrales}
C.~S. Chou, P.-A. Meyer, and C.~Stricker.
\newblock Sur les int{\'e}grales stochastiques de processus pr{\'e}visibles non born{\'e}s.
\newblock \emph{S{\'e}minaire de Probabilit{\'e}s}, 14:\penalty0 128--139, 1980.

\bibitem[Cohen and Elliott(2015)]{cohen15calculus}
S.~N. Cohen and R.~J. Elliott.
\newblock \emph{Stochastic Calculus and Applications}.
\newblock Birkh{\"a}user, 2nd edition, 2015.

\bibitem[Cybenko(1989)]{cybenko89approximation}
G.~Cybenko.
\newblock Approximation by superpositions of a sigmoidal function.
\newblock \emph{Mathematics of Control, Signals, and Systems}, 2:\penalty0 303--314, 1989.

\bibitem[Delbaen and Schachermayer(1994)]{delbaen94arbitrage}
F.~Delbaen and W.~Schachermayer.
\newblock A general version of the fundamental theorem of asset pricing.
\newblock \emph{Mathematische Annalen}, 300:\penalty0 463--520, 1994.

\bibitem[Dellacherie and Meyer(1980)]{dellacherie80probabilites}
C.~Dellacherie and P.-A. Meyer.
\newblock \emph{Probabilit{\'e}s et Potentiel. Chapitres {V} {\`a} {VIII}: Th{\'e}orie des Martingales}.
\newblock Hermann, 1980.

\bibitem[{ECB}(2019)]{ecb19algorithmic}
{ECB}.
\newblock Algorithmic trading: trends and existing regulation, 2019.

\bibitem[{\'E}mery(1979)]{emery79topologie}
M.~{\'E}mery.
\newblock Une topologie sur l'espace des semimartingales.
\newblock \emph{S{\'e}minaire de Probabilit{\'e}s}, 13:\penalty0 260--280, 1979.

\bibitem[{ESMA}(2026)]{esma26algorithmic}
{ESMA}.
\newblock Supervisory briefing on algorithmic trading in the {EU}, 2026.

\bibitem[Ethier and Kurtz(1986)]{ethier86markov}
S.~N. Ethier and T.~G. Kurtz.
\newblock \emph{Markov Processes: Characterization and Convergence}.
\newblock John Wiley \& Sons, 1986.

\bibitem[He et~al.(1992)He, Wang, and Yan]{he92semimartingale}
S.-W. He, J.-G. Wang, and J.-A. Yan.
\newblock \emph{Semimartingale Theory and Stochastic Calculus}.
\newblock CRC Press, 1992.

\bibitem[Hornik(1991)]{hornik91approximation}
K.~Hornik.
\newblock Approximation capabilities of multilayer feedforward networks.
\newblock \emph{Neural Networks}, 4\penalty0 (2):\penalty0 251--257, 1991.

\bibitem[Hornik et~al.(1989)Hornik, Stinchcombe, and White]{hornik89approximation}
K.~Hornik, M.~Stinchcombe, and H.~White.
\newblock Multilayer feedforward networks are universal approximators.
\newblock \emph{Neural Networks}, 2\penalty0 (5):\penalty0 359--366, 1989.

\bibitem[Hunt et~al.(1992)Hunt, Sauer, and Yorke]{hunt92prevalence}
B.~R. Hunt, T.~Sauer, and J.~A. Yorke.
\newblock Prevalence: {A} translation-invariant ``almost every'' on infinite-dimensional spaces.
\newblock \emph{Bulletin of the American Mathematical Society}, 27:\penalty0 217--238, 1992.

\bibitem[Kallenberg(2021)]{kallenberg21foundations}
O.~Kallenberg.
\newblock \emph{Foundations of Modern Probability}.
\newblock Springer, 3rd edition, 2021.

\bibitem[Karatzas and Shreve(1998)]{karatzas98calculus}
I.~Karatzas and S.~E. Shreve.
\newblock \emph{Brownian Motion and Stochastic Calculus}.
\newblock Springer, 2nd edition, 1998.

\bibitem[Kratsios(2021)]{kratsios21approximation}
A.~Kratsios.
\newblock The universal approximation property.
\newblock \emph{Annals of Mathematics and Artificial Intelligence}, 89:\penalty0 435--469, 2021.

\bibitem[Krickeberg(1964)]{krickeberg64conditional}
K.~Krickeberg.
\newblock Convergence of conditional expectation operators.
\newblock \emph{Theory of Probability \& Its Applications}, 9\penalty0 (4):\penalty0 538--549, 1964.

\bibitem[Mania et~al.(2008)Mania, Tevzadze, and Toronjadze]{mania08partial}
M.~Mania, R.~Tevzadze, and T.~Toronjadze.
\newblock Mean-variance hedging under partial information.
\newblock \emph{SIAM Journal on Control and Optimization}, 47\penalty0 (5):\penalty0 2381--2409, 2008.

\bibitem[{\O}ksendal(2003)]{oksendal03equations}
B.~{\O}ksendal.
\newblock \emph{Stochastic Differential Equations: An Introduction with Applications}.
\newblock Springer, 6th edition, 2003.

\bibitem[Pestman(1995)]{pestman95measurability}
W.~R. Pestman.
\newblock Measurability of linear operators in the {S}korokhod topology.
\newblock \emph{Bulletin of the Belgian Mathematical Society - Simon Stevin}, 2\penalty0 (4):\penalty0 381--388, 1995.

\bibitem[Protter(2005)]{protter05integration}
P.~E. Protter.
\newblock \emph{Stochastic Integration and Differential Equations}.
\newblock Springer, 2nd edition, 2005.

\bibitem[Rao and Ren(1991)]{rao91orlicz}
M.~M. Rao and Z.~D. Ren.
\newblock \emph{Theory of {Orlicz} Spaces}.
\newblock Marcel Dekker, 1991.

\bibitem[Rheinl{\"a}nder and Schweizer(1997)]{rheinlaender97projections}
T.~Rheinl{\"a}nder and M.~Schweizer.
\newblock On ${L}^{2}$-projections on a space of stochastic integrals.
\newblock \emph{The Annals of Probability}, 25\penalty0 (4):\penalty0 1810--1831, 1997.

\bibitem[Rudin(1987)]{rudin87analysis}
W.~Rudin.
\newblock \emph{Real and Complex Analysis}.
\newblock McGraw-Hill, 3rd edition, 1987.

\bibitem[Rudin(1991)]{rudin91functional}
W.~Rudin.
\newblock \emph{Functional Analysis}.
\newblock McGraw-Hill, 2nd edition, 1991.

\bibitem[Schilling(2017)]{schilling17measures}
R.~L. Schilling.
\newblock \emph{Measures, Integrals and Martingales}.
\newblock Cambridge University Press, 2nd edition, 2017.

\bibitem[Schmidhuber(2015)]{schmidhuber15overview}
J.~Schmidhuber.
\newblock Deep learning in neural networks: {A}n overview.
\newblock \emph{Neural Networks}, 61:\penalty0 85--117, 2015.

\bibitem[Schmock(2024)]{schmock24stochastic}
U.~Schmock.
\newblock Lecture notes in stochastic analysis for financial and actuarial mathematics.
\newblock TU Wien, 2024.

\bibitem[Schweizer(1994)]{schweizer94restricted}
M.~Schweizer.
\newblock Risk-minimizing hedging strategies under restricted information.
\newblock \emph{Mathematical Finance}, 4\penalty0 (4):\penalty0 327--342, 1994.

\bibitem[Schweizer(2001)]{schweizer01quadratic}
M.~Schweizer.
\newblock A guided tour through quadratic hedging approaches.
\newblock In \emph{Option Pricing, Interest Rates and Risk Management}, pages 538--574. Cambridge University Press, 2001.

\bibitem[Shiryaev(2008)]{shiryaev08stopping}
A.~N. Shiryaev.
\newblock \emph{Optimal Stopping Rules}.
\newblock Springer, 2008.

\bibitem[{Staff of the SEC}(2020)]{sec20algorithmic}
{Staff of the SEC}.
\newblock Staff report on algorithmic trading in {U.S.} capital markets, 2020.

\bibitem[{Staffs of the CFTC and SEC}(2010)]{cftcsec10market}
{Staffs of the CFTC and SEC}.
\newblock Findings regarding the market events of {May 6, 2010}, 2010.

\bibitem[Stricker(1990)]{stricker90arbitrage}
C.~Stricker.
\newblock Arbitrage et lois de martingale.
\newblock \emph{Annales de l'I.H.P. Probabilit{\'e}s et statistiques}, 26\penalty0 (3):\penalty0 451--460, 1990.

\bibitem[{The Royal Swedish Academy of Sciences}(2024)]{nobel24physics}
{The Royal Swedish Academy of Sciences}.
\newblock The {Nobel Prize in Physics} 2024: They trained artificial neural networks using physics, 2024.

\bibitem[Unser(2023)]{unser23ridges}
M.~Unser.
\newblock Ridges, neural networks, and the {Radon} transform.
\newblock \emph{Journal of Machine Learning Research}, 24:\penalty0 1--33, 2023.

\bibitem[Yan(1980)]{yan80caracterisation}
J.-A. Yan.
\newblock Caract{\'e}risation d'une classe d'ensembles convexes de {$L^1$} ou {$H^1$}.
\newblock \emph{S{\'e}minaire de Probabilit{\'e}s}, 14:\penalty0 220--222, 1980.

\end{thebibliography}

\end{document}